\documentclass[12pt,reqno]{amsart}
\pdfoutput=1

\usepackage{amsmath,amsfonts,amsthm,amssymb,amsxtra}
\usepackage{bbm} % to get \1

\usepackage{xfrac}
\usepackage{todonotes, relsize, pgfplots, tikz, float}
\usetikzlibrary{calc,decorations.markings}
\tikzset{
	> = stealth,
	every pin/.style = {pin edge = {}},
	flow/.style = {decoration = {markings, mark=at position #1 with {\arrow{>}}},
		postaction = {decorate}
	},
	flow/.default = 0.5,
	main/.style = {color=#1, line width=0.5pt, line cap=round, line join=round},
	main/.default = black,
	fontscale/.style={font=\relsize{#1}},
}

\newtheorem{theorem}{Theorem}
\newtheorem{proposition}[theorem]{Proposition}
\newtheorem{lemma}[theorem]{Lemma}
\newtheorem{corollary}[theorem]{Corollary}

\theoremstyle{definition}

\newtheorem{definition}[theorem]{Definition}

\theoremstyle{remark}

\numberwithin{equation}{section}

\renewcommand{\epsilon}{\varepsilon}

\newcommand{\N}{\mathbb{N}}

\renewcommand{\phi}{\varphi}
\newcommand{\R}{\mathbb{R}}

\newcommand{\Sph}{\mathbb{S}}

\DeclareMathOperator{\dist}{dist}

\DeclareMathOperator{\ran}{ran}

\usepackage{mathtools, todonotes, thmtools}\usepackage{tikz}
\usetikzlibrary{datavisualization}
\usetikzlibrary{datavisualization.formats.functions}
\usetikzlibrary{patterns}

\let\oldtocsection=\tocsection

\let\oldtocsubsection=\tocsubsection

\let\oldtocsubsubsection=\tocsubsubsection

\renewcommand{\tocsection}[2]{\hspace{0em}\oldtocsection{#1}{#2}}
\renewcommand{\tocsubsection}[2]{\hspace{1em}\oldtocsubsection{#1}{#2}}
\renewcommand{\tocsubsubsection}[2]{\hspace{2em}\oldtocsubsubsection{#1}{#2}}

\newif\ifSolution
\Solutiontrue %comment to remove solutions
\newif\ifNoSolution
\NoSolutiontrue
\ifSolution 
\NoSolutionfalse 
\fi
\newenvironment{solution}{\ifSolution\begin{proof}
		\else\setbox0\vbox\bgroup\fi}{\ifSolution \end{proof}\else\egroup\fi}

\usepackage{hyperref}	% to get hyperlinks in equations and references

\begin{document}

	\begin{titlepage}
		\huge \title[Degenerate Stability of the CKN-Inequality along the FS-Curve]{Degenerate Stability of the Caffarelli--Kohn--Nirenberg Inequality along the Felli--Schneider Curve} \setcounter{page}{1}
		\vspace{7cm}
	\end{titlepage}
	
	\author{Rupert L.~Frank}
	\author{Jonas W.~Peteranderl}
	
	\address[Rupert L.~Frank]{Mathe\-matisches Institut, Ludwig-Maximilians Universit\"at M\"unchen, The\-resienstr.~39, 80333 M\"unchen, Germany, and Munich Center for Quantum Science and Technology, Schel\-ling\-str.~4, 80799 M\"unchen, Germany, and Mathematics 253-37, Caltech, Pasa\-de\-na, CA 91125, USA}
	\email{r.frank@lmu.de}
	\address[Jonas W.~ Peteranderl]{Mathematisches Institut, Ludwig-Maximilians Universit\"at M\"unchen, Theresienstr.~39, 80333 M\"unchen, Germany}	\email{peterand@math.lmu.de}
	
	\thanks{\copyright\, 2023 by the authors. This paper may be reproduced, in its entirety, for non-commercial purposes.\\
		Partial support through US National Science Foundation grant DMS-1954995 (R.L.F.), as well as through the Deutsche Forschungsgemeinschaft through Germany’s Excellence Strategy EXC-2111-390814868 (R.L.F.) and through TRR 352 -- Project-ID 470903074 (R.L.F. \& J.W.P.) is acknowledged.}

	\begin{abstract}
		We show that the Caffarelli--Kohn--Nirenberg (CKN) inequality holds with a remainder term that is quartic in the distance to the set of optimizers for the full parameter range of the Felli--Schneider (FS) curve. The fourth power is best possible. This is due to the presence of non-trivial zero modes of the Hessian of the deficit functional along the FS-curve. Following an iterated Bianchi--Egnell strategy, the heart of our proof is verifying a `secondary non-degeneracy condition'. Our result completes the stability analysis for the CKN-inequality to leading order started by Wei and Wu. Moreover, it is the first instance of degenerate stability for non-constant optimizers and for a non-compact domain.
	\end{abstract}
	
	\maketitle
	
	\tableofcontents
	\newpage
	
	\pagenumbering{arabic}\setcounter{page}{2}
	\section{Introduction and main result}
	
	\subsection{The CKN-inequality}
	Caffarelli, Kohn, and Nirenberg \cite{CaKoNi} introduced, among others, the family of functional inequalities
	\begin{equation}\label{CKN}
		\int_{\R^d} \frac{|\nabla v|^2}{|x|^{2a}}\ \mathrm{d} x \geq C_{a,b}	\left(\int_{\R^d} \frac{|v|^q}{|x|^{bq}}\ \mathrm{d}x\right)^{\sfrac{2}{q}} ,
	\end{equation}
	nowadays known as CKN-inequality,
	with dimension $d\in\N$ and parameters $a,b\in\R$ such that \begin{equation}\label{adm}
		a<\operatorname{min}\{0,b\}+\frac{d-2}{2}
		\qquad\text{and}\qquad
		0\leq b-a\leq 1\,.
	\end{equation}
	We will call pairs $(a,b)$ satisfying \eqref{adm} \textit{admissible}. Note that the first condition reduces to $2a<d-2$ in case $d>2$. We use the convention that $C_{a,b}$ denotes the optimal constant in \eqref{CKN}. By a scaling argument, one can show that $q$ has to satisfy
	\begin{equation}\label{q}
		q=\frac{2d}{d-2+2(b-a)} \,.
	\end{equation}
	Note that for admissible $(a,b)$ the exponent $q$ ranges between $2$ and $2^*$. Here $2^*$ denotes the critical Sobolev exponent with $2^*=2d(d-2)^{-1}$ in case $d\geq 3$ and $2^*=\infty$ in case $d=1,2$. In fact, the CKN-inequality contains the classical Sobolev inequality ($a=b=0$, $d\geq3$) as well as the Hardy inequality ($a=0$, $b=1$, $d\geq 3$) as special cases. If $(a,b)$ is admissible, $v$ is allowed to be a function in $\mathcal D^{1}_a(\R^d)$, the completion of $C_c^\infty(\R^d)$ with respect to the norm
	\begin{equation*}
		\left(\int_{\R^d} \frac{|\nabla v|^2}{|x|^{2a}}\ \mathrm{d} x\right)^{\sfrac{1}{2}}.
	\end{equation*}
	Horiuchi \cite{Ho} ($d\geq 2$) and Catrina and Wang \cite{CaWa} ($d\geq 1$) were able to complete the existing results on whether the optimal constant for \eqref{CKN} is attained. Indeed, among all admissible $(a,b)$, an affirmative result was proved in case
	\begin{equation}\label{adm2}
		0<b-a<1 
		\qquad\text{or}\qquad
		b=a\geq0 \,,
	\end{equation}which is sharp. We will call admissible pairs $(a,b)$ satisfying \eqref{adm2} \textit{attainable} and denote the set of optimizers of \eqref{CKN} by
	\[\mathcal Z\coloneqq \{v\in \mathcal D^{1}_a(\R^d): \text{\eqref{CKN} becomes an equality}\}.\] 
	If we restrict \eqref{CKN} to radial functions, that is, functions that only depend on the radial coordinate, and call the corresponding optimal constant $ C_{a,b}^*$, then, obviously, $C_{a,b}\leq C_{a,b}^*$. We will call an admissible pair $(a,b)$ \textit{symmetric} if $C_{a,b}=C_{a,b}^*$. Otherwise, symmetry breaking is said to occur. The constant $C_{a,b}^*$ can be determined explicitly, and, for $a\leq b<a+1$, the set of radial optimizers is given by \begin{equation}\label{radopt}
		\left\{\lambda\frac{\mu^{\sqrt{\Lambda}}(2q\Lambda)^{\frac{1}{q-2}}}{(1+|\mu \cdot|^{\sqrt{\Lambda}(q-2)})^{\frac{2}{q-2}}}\right\}_{\mu>0, \lambda\in\R} \text{ with } \Lambda\coloneqq \left(\frac{d-2-2a}{2}\right)^2>0
	\end{equation} 
	(see \cite[p. 236 f.]{CaWa}). This set agrees with $\mathcal Z$ in case $(a,b)$ is symmetric \cite{DoEsLo2}. The only exception is $(a,b)=(0,0)$, where the set of optimizers contains, in addition, the translates of functions in \eqref{radopt}.
	
	Admissible pairs $(a,b)$ with $a\geq 0$ are well-known to be symmetric (see, e.g., \cite{Ta,Au}, \cite{Lieb}, and \cite{ChCh}), so, if symmetry breaking occurs, then necessarily $a<0$. The fact that symmetry breaking does occur for some parameters was observed by Catrina and Wang \cite{CaWa}. Thereafter, Felli and Schneider found an explicit curve that encloses a region where symmetry breaking occurs; see \cite[Corollary 1.2]{FeSc} for $d\geq 3$ and \cite[Theorem 1.1]{DoEsTa} for $d=2$.
	More specifically, among all attainable $(a,b)$, they showed that the pair $(a,b)$ is \textit{not} symmetric if
	\begin{equation}\label{symbreak}
		\Lambda>4 \, \frac{d-1}{q^2-4}\eqqcolon \Lambda_{FS} \,.
	\end{equation}
	Note that for fixed dimension $d$, $\Lambda= \Lambda_{FS}$ describes a curve in the $(a,b)$-plane since $q=q(a,b)$ and $\Lambda=\Lambda(a)$ are parametrized by $(a,b)$ and $a$, respectively. The condition \eqref{symbreak} is trivially satisfied if $d=1$, and hence symmetry is broken for all attainable $(a,b)$ in this case, which is in line with \cite[Theorem 7.2]{CaWa}. After various partial results \cite{SmWi, LiWa,DoEsLoTa, DoEsLo}, Dolbeault, Esteban, and Loss \cite[Theorem 1.1]{DoEsLo2} were able to settle the longstanding conjecture on the optimal symmetry range by proving that the value $\Lambda_{FS}$ indeed separates the symmetry from the symmetry breaking region. More concretely, among all attainable $(a,b)$, they proved that the pair $(a,b)$ is symmetric if and only if
	\begin{equation}\label{adm3}
		\Lambda\leq \Lambda_{FS}\,.
	\end{equation}
	In fact, they established that \textit{all} optimizers are radial in the symmetric case. Assuming $a<0$, the case of equality in \eqref{adm3} then determines the FS-curve. At this point, let us briefly mention that being attainable and being symmetric are not disjoint properties. Indeed, we have $C_{a,a+1}=C^*_{a,a+1}$ by continuity in case $d\geq 2$ \cite[Theorem 1.1(i), Theorem 7.5(i), Remark 3.4]{CaWa}, so $(a,b)$ is symmetric for $b=a+1$ but not attainable. (Note that the formula of $C_{a,b}^*$ given in \cite[Eq.~(2.13)]{CaWa} holds for $d=2$ as well.) If $d=1$, this follows from \cite[p. 254]{CaWa} as $C_{a,b}^*-C_{a,b}\to 0$ for $1+a-b\to 0$. \
	On the other hand, $C_{a,a}<C^*_{a,a}$ as computed in \cite[Proof of Theorem 1.3(ii)]{CaWa}, so admissible $(a,b)$ with $b=a$ are neither symmetric nor attainable.
	
	\subsection{Stability for the CKN-inequality -- our main result}
	
	In this paper we are interested in the question of stability, that is, whether the closeness to $1$ of the quotient of the two sides in \eqref{CKN} for some $v$ implies the closeness of $v$ to the set $\mathcal Z$. 
	
	For the Sobolev inequality ($a=b=0$, $d\geq3$), this question was raised by Brezis and Lieb \cite{BrLi}. Bianchi and Egnell \cite{BiEg} gave an affirmative answer with the gradient $L^2$-norm $\|\nabla \cdot\|_{L^2(\R^d)}$ as a measure for the distance to the set of optimizers. They showed that this distance vanishes at least quadratically in the difference between $1$ and the quotient of the two sides in \eqref{CKN}. Bianchi and Egnell introduced a very robust technique, which has been adapted to many other functional inequalities; see, for instance, \cite{RaSmWi} for the Hardy--Sobolev inequality or \cite{ChFrWe} for the fractional Sobolev inequality. This technique is based on two ingredients, namely, a compactness theorem for optimizing sequences and a spectral analysis around an optimizer. For a recent quantitative variant of the basic Bianchi--Egnell argument, leading to an optimal dependence of the stability constants, we refer to \cite{DoEsFiFrLo}. Further progress on related questions can be found in \cite{ChLuTa,Ko,Ko2}. For an introduction to the Sobolev inequality and its stability, see \cite{Fr2}. A quantitative version of the Hardy inequality ($a=0$, $b=1$, $d\geq 3$) appeared in \cite{CiFe}.
	
	For the CKN-inequality, a Bianchi--Egnell-type stability inequality was recently shown by Wei and Wu \cite{WeWu} in the interior of the symmetric regime ($\Lambda<\Lambda_{FS}$); see also \cite{RaSmWi} for an earlier contribution in case $a=0$. 
	
	Our main result is a stability inequality on the boundary of the symmetric regime, that is, on the FS-curve $\Lambda=\Lambda_{FS}$. Remarkably, while the Wei--Wu result in the interior of the symmetric regime involves a remainder term \emph{quadratic} in the distance to the set of optimizers, our bound will involve a remainder term that is \emph{quartic} in this distance. We will also show that this quartic vanishing is best possible. The reason is that in the spectral analysis part of the Bianchi--Egnell strategy additional zero modes appear, namely, zero modes that do not come from symmetries of the problem. [In passing, we note an inaccuracy in \cite{WeWu}; their quadratic stability result only holds in the parameter range \textit{excluding} the FS-curve, as the inequality \cite[Eq.~(4.4)]{WeWu} breaks down due to the existence of non-trivial zero modes. This was also noticed in \cite{DeTi2}.]
	
	\begin{theorem}[Degenerate stability of the CKN-inequality along the FS-curve]\label{thm}
		Let $(a,b)\in\R^2$ satisfy $a<0$ and $\Lambda=\Lambda_{FS}$ with $d\geq 2$ and $q$ given by \eqref{q}.
		Then there is a constant $ c(q,d)>0$ such that for all $v\in \mathcal D^{1}_a (\R^d)$,
		\begin{equation*}
			\int_{\R^d}\frac{|\nabla v|^2}{ |x|^{2a}}\ \mathrm{d} x-	C_{a,b}\left(	\int_{\R^d} \frac{|v|^q}{ |x|^{bq}}\ \mathrm{d} x\right)^{\sfrac{2}{q}}\geq  c(q,d)\inf_{\chi\in\mathcal Z}\left(\int_{\R^d}\frac{|\nabla (v-\chi)|^2}{ |x|^{2a}}\ \mathrm{d} x\right)^{2}\left(\int_{\R^d}\frac{|\nabla v|^2}{ |x|^{2a}}\ \mathrm{d} x\right)^{-1}.
		\end{equation*}
		Moreover, the inequality is best possible with respect to the quartic vanishing of the distance to $\mathcal Z$, that is, there is a sequence $(v_n)_n\subset\mathcal D^1_a(\R^d)\setminus\{0\}$ with 
		$$
		\lim_{n\to\infty} \frac{ \int_{\R^d}\frac{|\nabla v_n|^2}{ |x|^{2a}}\ \mathrm{d} x}{ \left(	\int_{\R^d} \frac{|v_n|^q}{ |x|^{bq}}\ \mathrm{d} x\right)^{\sfrac{2}{q} } } = C_{a,b}
		$$
		and
		$$
		\limsup_{n\to\infty} \frac{\int_{\R^d}\frac{|\nabla v_n|^2}{ |x|^{2a}}\ \mathrm{d} x-	C_{a,b}\left(	\int_{\R^d} \frac{|v_n|^q}{ |x|^{bq}}\ \mathrm{d} x\right)^{\sfrac{2}{q}}}{\inf_{\chi\in\mathcal Z}\left(\int_{\R^d}\frac{|\nabla (v_n-\chi)|^2}{ |x|^{2a}}\ \mathrm{d} x\right)^{2}\left(\int_{\R^d}\frac{|\nabla v_n|^2}{ |x|^{2a}}\ \mathrm{d} x\right)^{-1}} < \infty \,.
		$$
	\end{theorem}
	
	This theorem establishes the stability of the CKN-inequality along the FS-curve, but only in a degenerate sense, where the distance to the set of optimizers vanishes faster than quadratically; compare \eqref{van}. The interest in such degenerate stability of functional inequalities has been raised recently through a work by Engelstein, Neumayer, and Spolaor \cite{EnNeSp}, who investigated the quantitative stability of the Yamabe problem for closed Riemannian manifolds. While non-degenerate stability (with a square of the distance to the set of optimizers) was proved for manifolds of a generic type, it was also shown that the manifold
	\begin{equation}\label{SchoenEx}
		\mathbb S^1((d-2)^{-\sfrac{1}{2}})\times \mathbb S^{d-1}(1)
	\end{equation}
	with its standard product metric exhibits only degenerate stability, namely, with an (unspecified) power of the distance that is strictly larger than two. This example builds upon work by Schoen \cite{Sc}. Recently, it was shown \cite{Fr} that in the example \eqref{SchoenEx} the sharp stability exponent is four. The same phenomenon was observed in other types of Sobolev-type inequalities. Let us stress that the example \eqref{SchoenEx} indeed describes a degenerate scenario since the stability becomes non-degenerate -- that is, a stability inequality with a quadratic distance to the optimizers -- when varying the radius of the one-dimensional sphere in \eqref{SchoenEx}; see \cite{Fr2} for more details. Similarly, as we show in this paper, while non-degenerate stability was shown for the CKN-inequality in the interior of the symmetry region \cite{WeWu}, only the weaker notion of degenerate stability with a fourth power in the distance is available along the FS-curve. Therefore, our result proves a loss of stability and highlights the phase transition occurring due to symmetry breaking.
	
	The underlying mechanism for degenerate stability in \cite{Fr} and in the present paper is similar. It is caused by the presence of zero modes of the Hessian of the deficit functional that do not come from symmetries of the problem. As we will explain below, there are various features of the CKN-setting that make the present analysis substantially harder than the one in \cite{Fr}.
	
	We emphasize that the degeneracy along the FS-curve occurs only on a finite-dimensional subspace of $\mathcal D^{1}_a (\R^d)$, and hence an actual stronger stability result holds, with right side proportional to
	\begin{equation*}
		\inf_{\chi\in\mathcal Z}\left(\int_{\R^d}\frac{|\nabla (\Pi_d v-\chi)|^2}{ |x|^{2a}}\ \mathrm{d} x\right)^{2}\left(\int_{\R^d}\frac{|\nabla (\Pi_dv)|^2}{ |x|^{2a}}\ \mathrm{d} x\right)^{-1}+	\inf_{\chi\in\mathcal Z}\left(\int_{\R^d}\frac{|\nabla (\Pi_d^\perp v-\chi)|^2}{ |x|^{2a}}\ \mathrm{d} x\right),
	\end{equation*}
	where $\Pi_d$ is the orthogonal projection in $H^1(\mathcal C)$ onto the $d$-dimensional subspace of non-trivial zero modes and $\Pi_d^\perp\coloneqq1- \Pi_d$. (For the precise definition of non-trivial zero modes, we refer to Subsection \ref{subsec1.4}.) This follows by a slight modification of our proof, as in \cite{Fr2}. Such a mix of quadratic and quartic stability was first observed by Brigati, Dolbeault, and Simonov \cite{BrDoSi} in the setting of the log-Sobolev inequality on the sphere, which is yet another example of a degenerately stable functional inequality.
	
	Let us conclude this subsection by highlighting in which respect the present paper goes beyond the above mentioned works on degenerate stability. An obvious difference is that, in contrast to the inequalities covered in \cite{Fr}, the CKN-inequality contains integrals over a non-compact domain, and that the optimizers are non-constant functions. This leads to several technical complications -- the crucial one being the verification of a certain secondary non-degeneracy condition, which we will describe in detail below. This part of the proof, whose analogue in the case of constant minimizers in \cite{Fr} follows by a straightforward computation, is one of our main achievements here and takes up a significant part of this paper. It involves a series solution of a certain inhomogeneous second order equation and then the verification of certain positivity properties of an infinite series; see Section \ref{sec:solving} and the proof of Proposition~\ref{prop3}. We stress that our treatment is fully analytical and does not rely on numerical assistance. Finally, we want to stress a novel approach to deal with the quartic order expansion for the $L^q$-norm when $2<q<4$, which would simplify and unify the different ad-hoc approaches employed in \cite{Fr}. We hope this will be useful in other related works on degenerate stability.
	
	%%%%%%%%%%%%%%%%%%%%%%%%%%%%%%	
	
	\subsection{A reformulation} \label{subsec1.2}
	As is common practice, we employ logarithmic coordinates to transform \eqref{CKN} to a Sobolev inequality on the cylinder $\mathcal{C}\coloneqq\R\times \mathbb S^{d-1}$ without weights. With the so-called Emden--Fowler transformation
	\begin{equation*}
		v(r,\omega)=r^{a-\frac{(d-2)}{2}}  \phi(s,\omega)\,,
	\end{equation*}
	where $r=|x|$, $s=-\log r$, and $\omega=x/r$, we can write \eqref{CKN} as 
	\begin{equation}\label{CKNEF}
		\|\partial_s \phi\|^2_{L^2(\mathcal{C})}+\|\nabla_\omega \phi\|^2_{L^2(\mathcal{C})}+\Lambda\|\phi\|^2_{L^2(\mathcal{C})}\geq C_{a,b} \| \phi\|^2_{L^q(\mathcal{C})}
	\end{equation}
	for $\phi\in H^1(\mathcal{C})$. Here $\nabla_\omega$ denotes the gradient on $\mathbb S^{d-1}$. Via logarithmic variables the scaling invariance of the CKN-inequality turns into the translation invariance of the Sobolev inequality on the cylinder. Note that \eqref{CKNEF} is an equality if and only if \eqref{CKN} is, and, calling functions on $\mathcal C$ that depend only on $s$ \emph{radial}, the results on symmetry and symmetry breaking carry over to \eqref{CKNEF} as well.
	
	For symmetric, attainable $(a,b)\neq (0,0)$, equality in \eqref{CKNEF} is attained if and only if $\phi$ equals (up to a scalar multiple and a translation) the radial function
	\begin{equation*}
		u\coloneqq \beta (\cosh(\alpha\, \ \cdot \ ))^{-\frac{2}{q-2}}\,,\qquad \alpha\coloneqq \frac{q-2}{2} \sqrt{\Lambda}\,,\qquad\beta\coloneqq \left(\frac{q}{2}\Lambda\right)^{\frac{1}{q-2}};
	\end{equation*} see \cite[Cor.~1.3]{DoEsLo2} for a reference. By the Emden--Fowler transformation, the set $\mathcal Z$ is cast to
	\[\mathcal{M}\coloneqq \{\lambda u(\ \cdot - t)\}_{\lambda,t\in\R}\,,\]
	the set of optimizers for \eqref{CKNEF}. Some authors neglect the non-regular value $\lambda=0$ to obtain a differentiable manifold \cite{BiEg} or drop the multiplication by a scalar multiple entirely \cite{FeSc}. For our purposes, it is convenient to keep the full set of optimizers when deriving orthogonality relations.
	
	The values $\alpha$ and $\beta$ in the definition of $u$ are chosen such that $u$ is the unique even, positive function solving the Euler--Lagrange equation 
	\begin{equation}\label{EL}
		-\partial_s^2  u+ \Lambda u =u^{q-1}
	\end{equation} 
	on $\mathcal C$. We will use this and the following related equations frequently:
	\begin{align}
		\left( -\partial_s^2+ \Lambda-(q-1)u^{q-2} \right) \partial_s u &=0\,,\label{EL1}\\
		\left(-\partial_s^2 + \frac{q^2}{4}\Lambda-(q-1)u^{q-2}\right) u^{\sfrac{q}{2}} &=0\,.\label{EL2}
	\end{align}
	
	In the following, the norm in $L^q(\mathcal C)$ is denoted $\|\cdot\|_q$. In $H^1(\mathcal C)$ we use the ($\Lambda$-dependent) norm
	$$
	\|\cdot\| \coloneqq \left( \|\partial_s\cdot\|_2^2+\|\nabla_\omega \cdot\|_2^2+\Lambda \|\cdot\|_2^2 \right)^{\sfrac12}.
	$$
	The inner products in $L^2(\mathcal C)$ and $H^1(\mathcal C)$ are $\langle\cdot,\cdot\rangle_2$ and $\langle\cdot,\cdot\rangle$, respectively. Moreover, we will consider $a$, $b$, and $d$ satisfying \eqref{adm}, \eqref{adm2}, and \eqref{adm3}, that is, $(a,b)$ is admissible, attainable, and symmetric. As an immediate consequence, we know that $2<q<2^*$ and $d\geq 2$. In particular, the assumption $d\geq2$ in Theorem \ref{thm} is redundant and for clarity only. As the Sobolev inequality admits an additional symmetry, we exclude $(a,b)=(0,0)$.
	These assumptions on $a$, $b$, $d$, $q$, and $\Lambda$ are standard and will be used throughout Section \ref{sec2}, \ref{sec3}, and \ref{sec4}. In the latter two sections we will restrict ourselves to the FS-curve, that is, $\Lambda=\Lambda_{FS}$ for $a<0$.
	
	%%%%%%%%%%%%%%%%%%%%%%%%
	
	\subsection{Strategy of the proof}\label{subsec1.4}
	
	We will prove Theorem \ref{thm} in the equivalent formulation on the cylinder presented in the previous subsection. This will appear in Corollary \ref{cor} below. We will also present some of the main ingredients that go into the proof of this corollary. Those are stated in three propositions, whose proofs will be given in the remaining sections of this paper.
	
	Our basic technique in this paper will be the iterated Bianchi--Egnell strategy introduced in \cite{Fr}: While Bianchi and Egnell project on the space of trivial zero modes of the Hessian of the \textit{deficit functional} \[\mathcal F(\phi)\coloneqq \|\phi\|^2- C_{a,b} \| \phi\|^2_{q}\,, \qquad \phi\in H^1(\mathcal C)\,,\] it is possible to project further on the nearest non-trivial zero mode. This leads to a Taylor-type expansion of the deficit to quartic order of the distance to the set of optimizers
	\[\operatorname{dist}(\phi,\mathcal M)\coloneqq \inf_{\chi\in\mathcal M}\|\phi-\chi\|\,,\qquad \phi\in H^1(\mathcal C)\,.\]
	
	In our first step, we consider a minimizing sequence $(u_n)_n$ for the functional inequality \eqref{CKNEF} and project it on the nearest trivial zero mode. As the CKN-inequality \eqref{CKN} is invariant under dilations, and hence \eqref{CKNEF} under translations, we will have to handle the emerging lack of compactness. The content of our first proposition is that the projection can be chosen to be orthogonal in $H^1(\mathcal C)$ to \[
	\operatorname{span}\{u,\partial_s u\}\,, \] 
	the \textit{trivial} zero modes of the Hessian of $\mathcal F$ at $u$. We call these zero modes trivial, because they come from symmetries of $\mathcal F$, namely, from multiplication by a constant and from translations. The proposition will lead us to a decomposition of $(u_n)_n$ into an optimizer and a remainder term that converges to $0$ in the $H^1(\mathcal C)$-norm. 
	
	\begin{proposition}[Projection on the trivial zero modes of the Hessian]\label{prop1}
		Let $(a,b)\in\R^2\setminus \{(0,0)\}$ be admissible \eqref{adm}, attainable \eqref{adm2}, and symmetric \eqref{adm3} with $d\geq 2$ and $q$ given by \eqref{q}. Let $(u_n)_n$ be a sequence in $H^1(\mathcal C)$ such that 
		\begin{equation}
			\label{eq:optseq}
			\|u_n\|^2_q\to1\qquad \text{ and }\qquad \|u_n\|^2\to C_{a,b} \qquad\text{ for } n\to\infty\,.
		\end{equation}
		Then there are $\lambda_n\in\R\setminus\{0\}$, $ t_n\in\R$, and $r_n\in H^1(\mathcal{C})$ such that, along a subsequence, we have
		\begin{equation}\label{dec2}
			u_n(s,\omega)=\lambda_n (u+r_n)(s-t_n,\omega)\,,
			\qquad (s,\omega)\in\mathcal C\,,
		\end{equation}
		with $\operatorname{dist}(u_n,\mathcal M)=\|u_n-\lambda_n u(\ \cdot-t_n)\|=\|\lambda_n r_n\|$ and the following convergence and orthogonality properties.
		\begin{itemize}
			\item	\textit{Convergence properties}:  $\|r_n\|\to 0$ and $\lambda_n\to \lambda^*$
			hold for $n\to\infty$ and some $\lambda^*\in\R\setminus\{0\}$ with $|\lambda^*|= \|u\|^{-1}_q$.
			\item 	\textit{Orthogonality properties}: For all $n\in\N$ we have
			\begin{equation}\notag
				\langle r_n, u \rangle=\langle r_n, \partial_s u\rangle =0\,.
			\end{equation}
		\end{itemize}
	\end{proposition}
	
	We have written the orthogonality conditions in terms of the $H^1(\mathcal C)$-inner product. Using the equations \eqref{EL} and \eqref{EL1}, we can rewrite them in terms of the $L^2(\mathcal C)$-inner product as
	\begin{equation}\label{ortho1}
		\langle r_n, u^{q-1} \rangle_2=\langle r_n, (q-1)u^{q-2} \partial_s u\rangle_2=0\,.
	\end{equation}
	
	For $\Lambda<\Lambda_{FS}$, the Hessian of the deficit functional $\mathcal F$ at $u$ has trivial zero modes only. In case $\Lambda=\Lambda_{FS}$, however, the Hessian of $\mathcal F$ admits \textit{non-trivial} zero modes as well \cite{FeSc}, that is, the kernel of the Hessian strictly contains $\operatorname{span}\{u,\partial_s u\}$. While some authors refer to all zero modes that are not trivial as non-trivial zero modes, we will call a zero mode non-trivial if it lies in the orthogonal complement of the trivial zero modes. As we will see later, the space of non-trivial zero modes is given by the span of
	\[
	u^{\sfrac{q}{2}}\omega_i,
	\qquad
	i= 1,\dots,d \,,
	\]
	where $\omega_1,\dots,\omega_d$ denote the Cartesian coordinates restricted to $\mathbb S^{d-1}$. These modes do not arise from symmetries of $\mathcal F$.
	
	We will now see that, if we require the functional $\mathcal F$ to decay faster than the distance to the set of optimizers squared,
	\begin{equation}
		\label{van}	\lim_{n\to\infty} \frac{\mathcal F(u_n)}{\operatorname{dist}(u_n,\mathcal M)^2}=0\,,
	\end{equation}
	then a non-trivial zero mode contributes to $u_n$ satisfying \eqref{eq:optseq}, and we are able to further expand the minimizing sequence $(u_n)_n$. It turns out that we can decompose the previous remainder $r_n$ into the nearest non-trivial zero mode of the Hessian of $\mathcal F$ with decaying amplitude and a new remainder term that converges even faster in the $H^1(\mathcal C)$-norm. This new remainder can be chosen to be $H^1(\mathcal C)$-orthogonal to all zero modes of the Hessian.
	
	\begin{proposition}[Projection on the non-trivial zero modes of the Hessian]\label{prop2}
		Let $(a,b)\in\R^2$ satisfy $a<0$ and $\Lambda=\Lambda_{FS}$ with $d\geq 2$ and $q$ given by \eqref{q}. Let $(u_n)_n$ be a sequence in $H^1(\mathcal C)$ satisfying \eqref{eq:optseq} and \eqref{van}.	Then there are $\lambda_n\in\R\setminus\{0\}$, $t_n,\mu_n\in\R$, $D_n\in \mathcal{O}(d)$,
		and $R_n\in H^1(\mathcal{C})$ such that, along a subsequence, we have
		\begin{equation}\label{dec3}
			u_n(s,\omega)=\lambda_n (u+\mu_n(u^{\sfrac{q}{2}}\omega_d+R_n))(s-t_n, D_n\omega)\,, \qquad (s,\omega)\in \mathcal C\,,
		\end{equation}
		with $\lambda_n$, $t_n$, and $u$ from Proposition \ref{prop1} and the following additional convergence and orthogonality properties.
		\begin{itemize}
			\item	\textit{Convergence properties}: $\mu_n \to 0$ and $\|R_n\|\to 0$ hold for $n\to\infty$.
			\item 	\textit{Orthogonality properties:} For all $n\in\N$ we have 
			\begin{equation*}
				\langle R_n, u\rangle=\langle R_n, \partial_s u\rangle=\langle R_n, u^{\sfrac{q}{2}}\omega_i\rangle=0 \qquad\text{ for } i=1, ..., d\,.
			\end{equation*}
		\end{itemize}	Here $\mathcal O(d)$ denotes the set of orthogonal $d\times d$ matrices.
	\end{proposition}
	Using the equations \eqref{EL}, \eqref{EL1}, and \eqref{EL2}, we can rewrite the orthogonality conditions as
	\begin{equation}
		\label{eq:prop1orthol2}
		\langle R_n, u^{q-1} \rangle_2=\langle R_n, (q-1)u^{q-2} \partial_s u\rangle_2=\langle R_n, (q-1)u^{q-2}  u^{\sfrac{q}{2}}\omega_i
		\rangle_2=0 \qquad\text{ for } i=1, ..., d\,.
	\end{equation}
	
	It is not hard to show that, in fact, \eqref{van} and $\|u_n\|^2\to C_{a,b}$ for $n\to \infty$ make the assumption $\|u_n\|_q\to 1$ redundant.
	
	The decomposition in Proposition \ref{prop2} allows us to expand $\mathcal F$ to quartic order. In this way, we will obtain the following crucial asymptotic inequality, which will imply our main theorem.
	
	\begin{proposition}[Non-vanishing of the quartic order]\label{prop3}
		Let $(a,b)\in\R^2$ satisfy $a<0$ and $\Lambda=\Lambda_{FS}$ with $d\geq 2$ and $q$ given by \eqref{q}. There is an explicit constant $$J(q,d)>0\,,$$ which is given in \eqref{bestconst}, such that for every sequence $(u_n)_n\subset H^1(\mathcal C)$ satisfying \eqref{eq:optseq} we have
		\begin{equation}\label{finalest}
			\liminf_{n\to\infty} \frac{\|u_n\|^2\mathcal F(u_n)}{\operatorname{dist}(u_n,\mathcal M)^4}\geq J(q,d)\,.
		\end{equation}
		Moreover, the bound \eqref{finalest} is best possible in the sense that there is a sequence $(u_n)_n$ satisfying \eqref{eq:optseq} for which equality is attained in \eqref{finalest}.
	\end{proposition}
	
	The key point of this proposition is the strict inequality $J(q,d)>0$. In the limit $q\to 2^*$, it can be shown that $J(q,d)$ vanishes, which is due to the additional translation symmetry in case $(a,b)=(0,0)$. The non-vanishing of $J(q,d)$ for $q<2^*$ can be viewed as a \textit{secondary non-degeneracy condition} \cite{Fr,Fr2}. We give a concrete definiton in terms of variational derivatives of the deficit $\mathcal F$, which may be generalized to the setting of other functional inequalities.
	
	\begin{definition}\label{def:secnondeg}
		Let $\psi\in\mathcal M$. We say that the CKN-inequality satisfies the secondary non-degeneracy condition if
		\begin{equation}\label{secnondeg}
			(\partial_\epsilon^4 \mathcal F) (\psi+\epsilon (g+\epsilon \phi))|_{\epsilon=0}>0
		\end{equation}
		for every non-trivial zero mode $g$ in $ \operatorname{Ker}(D^2_\psi\mathcal F)$ and every $\phi$ that is $H^1(\mathcal C)$-orthogonal to  $\operatorname{Ker}(D^2_\psi\mathcal F)$. 
	\end{definition}
	
	Note that
	\begin{equation}\label{secnondeg1}
		(\partial_\epsilon^4 \mathcal F) (\psi+\epsilon (g+\epsilon \phi))|_{\epsilon=0}=12(D_\psi^2\mathcal F(\phi,\phi)+ D_\psi^3\mathcal F(g,g,\phi))+D_\psi^4\mathcal F (g,g,g,g)\,,
	\end{equation}
	where we wrote the differentials in $\psi$ as multilinear forms.
	
	As a consequence of the previous proposition, the secondary non-degeneracy condition can be verified, and we can prove, by contradiction, degenerate stability of quartic order. The following assertion is equivalent to Theorem \ref{thm} via the Emden--Fowler transformation.
	
	\begin{corollary}[Degenerate stability of a Sobolev inequality for a cylinder along the FS-curve] \label{cor}
		Let $(a,b)\in\R^2$ satisfy $a<0$ and $\Lambda=\Lambda_{FS}$ with $d\geq 2$ and $q$ given by \eqref{q}. Then there is a constant $ c(q,d)>0$ such that for all $\phi\in H^1(\mathcal{C})$,
		\begin{equation*}%\label{degstab}
			\mathcal F(\phi)\geq  c(q,d)\frac{\operatorname{dist}(\phi,\mathcal M)^4}	{\|\phi\|^2}\,.
		\end{equation*}
		Moreover, the inequality is best possible with respect to the power four, that is, there is a sequence $(\phi_n)_n\subset H^1(\mathcal C)\setminus\{0\}$ satisfying \eqref{eq:optseq} with
		$$
		\limsup_{n\to\infty} \frac{\|\phi_n\|^2\, \mathcal F(\phi_n)}{\operatorname{dist}(\phi_n,\mathcal M)^4} < \infty \,. 
		$$
	\end{corollary}
	
	\begin{solution}
		We argue by contradiction. For fixed $(q,d)$, assume there is a sequence $(u_n)_n \subset H^1(\mathcal C)$ with 
		\[\lim_{n\to\infty} \frac{\|u_n\|^2(\|u_n\|^2-C_{a,b}\|u_n\|_q^2)}{\inf_{\chi\in\mathcal M}\|u_n-\chi\|^4}= 0\,.\]
		By homogeneity, we may assume $\|u_n\|^2=C_{a,b}$. Using the CKN-inequality \eqref{CKNEF} for $u_n$ and the notion of the infimum $\inf_{\chi\in\mathcal M}\|u_n-\chi\|\leq\|u_n\|$, we find that
		\[0\leq \liminf_{n\to \infty} \left(1-\|u_n\|_q^2\right)\leq \liminf_{n\to\infty} \frac{\|u_n\|^2(\|u_n\|^2-C_{a,b}\|u_n\|_q^2)}{\inf_{\chi\in\mathcal M}\|u_n-\chi\|^4}=0\,.\]
		Hence, we have proved the required convergence properties for $(u_n)_n$ to apply Proposition \ref{prop3}, which leads to a contradiction.
		
		The assertion that the stability inequality is best possible with respect to the power four follows immediately from the corresponding assertion in Proposition \ref{prop3}. In fact, this shows that the sequence can be chosen such that the $\limsup$ in the assertion equals $J(q,d)$.
	\end{solution}
	
	If \eqref{secnondeg} did not hold, we would go on projecting on zero modes of $(\partial_\epsilon^{4} \mathcal F) \left(\psi+\epsilon (g+\epsilon \phi)\right)|_{\epsilon=0}$ and could formulate a similar next order non-degeneracy condition. Repeating this procedure, we would derive a non-degeneracy condition of higher order in every step. If one of these conditions was satisfied, the iteration scheme would end, and we would obtain a degenerate stability result with some exponent greater than four. Following \cite{EnNeSp}, one could probably show that this procedure terminates after finitely many steps. We are not aware of an example of a degenerate stability result with a distance to the power six or higher. 
	
	The remainder of this paper consists of four sections. In Section \ref{sec2}, \ref{sec3}, and \ref{sec4} we present the proofs of Proposition \ref{prop1}, \ref{prop2}, and \ref{prop3}, respectively. In Section \ref{sec:solving} we provide the details of some results used in Section \ref{sec4}.

	%%%%%%%%%%%%%%%%%%%%%%%%%%%%%%%
	%%%%%%%%%%%%%%%%%%%%%%%%%%%%%%%	
	
	\section{Projection on the trivial zero modes of the Hessian}\label{sec2}
	
	\subsection{Proof of Proposition \ref{prop1}}	
	As we are dealing with the non-compact domain $\mathcal C=\R\times\mathbb S^{d-1}$, we can only expect relative compactness of optimizing sequences to hold up to non-compact symmetries. Following Lions' concentration compactness principle \cite{Li}, we only have to rule out two phenomena -- \textit{vanishing} and \textit{dichotomy} -- in order to find convergent subsequences up to symmetries. For absence of vanishing we refer to \cite[Lemma 4.1]{CaWa} and references therein. Exploiting the Hilbert space structure of $H^1(\mathcal C)$, dichotomy can be excluded in a standard way; see \cite[Proof of Theorem 1.2(i)]{CaWa}, for instance. As a result, given $(u_n)_n\subset H^1(\mathcal C)$ such that $\|u_n\|_q^2\to 1$ and $\|u_n\|^2\to C_{a,b}$, there are $(t_n^*)_n\subset\R$ such that $u_n(\ \cdot +t_n^*, \cdot \ )$ converges in $H^1(\mathcal C)$. The limit is necessarily an optimizer of \eqref{CKNEF}. In addition, the limit has unit $L^q(\mathcal C)$-norm. According to \cite{DoEsLo2}, the limit is a translate of $\lambda^* u$, where $\lambda^*\in\R\setminus\{0\}$ satisfies $|\lambda^*|= \|u\|^{-1}_q$. By redefining the $t_n^*$, we may assume that the limit is $\lambda^* u$. Consequently, we can write
	\begin{equation}\label{dec}
		u_n(s,\omega)=\lambda^* (u+r_n^*)(s-t_n^*,\omega)\,,
		\qquad (s,\omega)\in\mathcal C\,,
	\end{equation}
	with $r^*_n\in H^1(\mathcal{C})$ satisfying $\|r_n^*\|\to 0$ for $n\to\infty$.
	
	Note that $t\mapsto \langle u(\ \cdot-t),u_n\rangle^2$ is a continuous, non-negative function that vanishes at infinity. Therefore, it attains its maximum at some $t_n\in\R$. Moreover, for any $t\in\R$, $\inf_\lambda	\|u_n- \lambda u(\ \cdot -t)\|^2 = \|u_n\|^2 - \|u\|^{-2} \langle u(\ \cdot-t),u_n\rangle^2$, and the infimum is attained at $\lambda=\|u\|^{-2} \langle u(\ \cdot-t),u_n\rangle$. Thus, if we set $\lambda_n:=\|u\|^{-2} \langle u(\ \cdot-t_n),u_n\rangle$, we see that $\operatorname{dist}(u_n,\mathcal M)^2= \inf_{\lambda,t} \|u_n- \lambda u(\ \cdot -t)\|^2$ is attained at $(\lambda_n,t_n)$.
	
	Define $\tilde r_n\coloneqq u_n(\ \cdot+t_n, \cdot \ )- \lambda_nu\in H^1(\mathcal C)$, and note that 
	\[
	\|\tilde r_n\|=\inf_{\chi\in\mathcal M}\|u_n-\chi\|\leq\|u_n-\lambda^* u(\ \cdot-t_n^*)\|=\|r_n^*\|=o_{n\to\infty}(1)\,.
	\]
	We now deduce that $\lambda_n\to\lambda^*$ and $t_n-t_n^*\to 0$. To this end, we first notice that
	$$
	|\|u_n\|_q - |\lambda_n| \|u\|_q | \leq \| u_n(\ \cdot+t_n,\cdot\ )-\lambda_n u\|_q = \|\tilde r_n\|_q \leq C_{a,b}^{-\sfrac12} \|\tilde r_n\| = o_{n\to\infty}(1)
	$$
	and $\|u_n\|_q \to 1$ imply that $|\lambda_n|\to \|u\|_q^{-1} = |\lambda^*|$. Next, we use the decomposition \eqref{dec} of $u_n$ to find that
	$$
	\lambda_n = \|u\|^{-2} \langle u(\ \cdot-t_n),u_n\rangle = \lambda^* \|u\|^{-2} \langle u(\ \cdot - t_n),u(\ \cdot - t_n^*) \rangle + o_{n\to\infty}(1) \,.
	$$
	Taking the absolute value on both sides, we deduce that $| \langle u(\ \cdot - t_n),u(\ \cdot - t_n^*) \rangle| \to \|u\|^2$. Using $\langle u(\ \cdot - t_n),u(\ \cdot - t_n^*)\rangle = \langle u(\ \cdot - t_n), u(\ \cdot - t_n^*)^{q-1}\rangle_{L^2(\mathcal C)}	\geq 0$ (by \eqref{EL} for $u(\ \cdot- t_n^*)$ and the fact that $u\geq 0$), we deduce on the one hand that $\lambda_n\to\lambda^*$ and on the other hand that
	$$
	\| u - u(\ \cdot+t_n^*-t_n) \|^2 = 2\left( \|u\|^2 - \langle u(\ \cdot - t_n),u(\ \cdot - t_n^*) \rangle \right) \to 0 \,.
	$$
	Since $u$ is symmetric decreasing, we deduce that $t_n^*-t_n\to 0$, as claimed.
	
	Since $\lambda_n\to \lambda^*\neq 0$, up to dropping finitely many $n$, we may assume $\lambda_n\not =0$. Defining $r_n\coloneqq \lambda_n^{-1}\tilde r_n$ then provides the desired decomposition \eqref{dec2} with $\|r_n\|\to 0$. 
	
	Note that $J_n(t,\lambda) \coloneqq\|u_n - \lambda u(\ \cdot - t)\|^2$ inherits the differentiability in $t$ from $u$. Moreover, it is a polynomial in $\lambda$ and hence differentiable in $\lambda$. This yields the desired orthogonality relations 
	\begin{align*}
		0&=\partial_\lambda J_n(t,\lambda)|_{(t,\lambda)=(t_n,\lambda_n)}&&\hspace{-0.4cm}=-2\lambda_n \langle r_n, u\rangle\,,\\ \hskip0.2878\textwidth 0&=\partial_t J_n(t,\lambda)|_{(t,\lambda)=(t_n,\lambda_n)}&&\hspace{-0.4cm}= 2\lambda_n^2\langle r_n, \partial_s u\rangle\,. \hskip0.249\textwidth 	\qed	\end{align*}
	
	%%%%%%%%%%%%%%%%%%%%%%
	
	\section{Projection on the non-trivial zero modes of the Hessian}\label{sec3}
	
	While the analysis in the previous section is relevant in the full range of admissible, attainable, and symmetric parameters, we now turn to properties that are specific for parameter values on the Felli--Schneider curve $\Lambda=\Lambda_{FS}$. Recall that $\omega_1,\dots,\omega_d$ denote the Cartesian coordinates restricted to $\mathbb S^{d-1}$. These generate the space of spherical harmonics of degree $1$. More generally, $(Y_{l,m})_m$ is defined to be the $L^2(\mathbb S^{d-1})$-orthonormal basis of
	spherical harmonics of degree $l$. The degeneracy index $m$ runs through a finite, $l$-dependent set, but we will not need a more detailed description for our purposes. Note that spherical harmonics of degree $0$ are constant. For an introduction to spherical harmonics, we refer to \cite[p. 137--152]{StWe}.
	
	\subsection{Degeneracy along the Felli--Schneider curve}
	\label{subsec3.1}
	
	Let $\psi=\lambda u(\ \cdot-t)\in\mathcal M$. We investigate the stability of the functional $\mathcal F$ around $\psi$ in the classical way by determining the zeros of the Hessian of $\mathcal F$. The Euler--Lagrange equation \eqref{EL} implies $C_{a,b}=\|u\|^2 \|u\|_q^{-2} =\|u\|_q^{q-2}$, which can be used to compute the Hessian of $\mathcal F$. One finds that for all $\phi\in H^1(\mathcal C)$,
	\begin{align*}
		&D^2_\psi \mathcal F (\phi)=\partial^2_\epsilon \mathcal F (\psi+\epsilon \phi)|_{\epsilon=0}
		\\&=2 \left(\|\phi\|^2-(q-1)\int_{\mathcal C}u(\ \cdot-t)^{q-2}\phi^2\ \mathrm{ d}(s,\omega)+(q-2)\|u\|_q^{-q}\left(\int_{\mathcal C}u(\ \cdot-t)^{q-1}\phi\ \mathrm{ d}(s,\omega)\right)^2\right).
	\end{align*}
	This quadratic form corresponds to a self-adjoint, lower bounded operator $\mathcal L_\psi$ in the Hilbert space $L^2(\mathcal C)$ with form domain $H^1(\mathcal C)$ and operator domain $H^2(\mathcal C)$ in the sense that
	$$
	D^2_\psi \mathcal F (\phi) =2 \langle\phi, \mathcal L_{\psi} \phi\rangle_2\,,
	$$
	where 
	\[
	\mathcal L_\psi\coloneqq -\partial_s^2-\Delta_\omega+\Lambda-(q-1)u(\ \cdot -t)^{q-2}+(q-2)\|u\|^{-q}_q |u^{q-1}(\ \cdot-t)\rangle\langle u^{q-1}(\ \cdot-t)| \,.
	\]
	Here $\Delta_\omega$ denotes the Laplace--Beltrami operator on $\mathbb S^{d-1}$, and $ |u^{q-1}(\ \cdot-t)\rangle\langle u^{q-1}(\ \cdot-t)|$ denotes the rank one projector onto $u^{q-1}(\ \cdot-t)$ in $L^2(\mathcal C)$. We stress that the inner product in the definition of the rank one projector is the one in $L^2(\mathcal C)$, not in $H^1(\mathcal C)$. We observe that the operator $\mathcal L_\psi$ is independent of $\lambda$, and hence $\mathcal L_\psi =\mathcal L_{u(\ \cdot  \,- t)}$.
	
	Note that $D^2_\psi\mathcal F$, and hence $\mathcal L_\psi$, is positive semi-definite by optimality of $\psi$. Indeed, we find $\mathcal F(\psi)=0$ and, through the Euler--Lagrange equation \eqref{EL}, $D_\psi\mathcal F=0$. Therefore, expanding $\mathcal F$ with $q>2$ around $\psi$ yields
	\[0\leq\mathcal F(\psi+\epsilon\phi)= \frac{\epsilon^2}{2}D^2_\psi\mathcal F(\phi)+o_{\epsilon\to 0}(\epsilon^2)\,, \qquad \phi\in H^1(\mathcal C)\,.\]
	
	Next, we show that the kernel of $\mathcal L_\psi$ is given by
	\begin{equation}	\label{KerHess}
		\operatorname 
		{Ker} \mathcal L_\psi  = \operatorname 
		{Ker} (D^2 _\psi\mathcal F)=\operatorname{span}\{ \psi, \partial_s \psi, \psi^{\sfrac{q}{2}}\omega_i, i=1,\dots, d \}\,.
	\end{equation}
	By means of $-\Delta_\omega\omega_i=(d-1)\omega_i$ and the equations \eqref{EL}, \eqref{EL1}, and \eqref{EL2}, it can be verified easily that $\mathcal L_\psi$ vanishes on $\psi, \partial_s \psi, \psi^{\sfrac{q}{2}}\omega_i, i=1,\dots, d$, which are mutually orthogonal in both $L^2(\mathcal C)$ and $H^1(\mathcal C)$. The other direction follows from a computation by Felli and Schneider \cite{FeSc}, which we briefly review here. After separating the radial and the angular part of an arbitrary solution $\phi\in H^1(\mathcal C)$ to $\mathcal L_\psi\phi=0$, they reduced this equation to an eigenvalue problem involving a one-dimensional Schr\"odinger operator with P\"oschl--Teller potential:
	\begin{equation}
		(-\partial_s^2-(q-1)u(\ \cdot -t)^{q-2})\Phi_l=\theta_l \Phi_l \label{ELQPT}
	\end{equation} 
	with $\theta_l\coloneqq-(l(l+d-2)+\Lambda)$ and $\Phi_l\in H^1(\R)$ for every $l\in\N_0$. The parameter $l$ corresponds to the angular momentum, that is, the degree of the spherical harmonic in the expansion of $\phi$.
	
	We make use of the following facts about the operator appearing in \eqref{ELQPT}.
	
	\begin{lemma}[Spectral analysis of lower eigenvalues]
		\label{pt}
		The lowest eigenvalue of the operator $-\partial_s^2-(q-1)u(\ \cdot -t)^{q-2}$ in $L^2(\R)$ is $-\frac{q^2}{4}\Lambda$ with corresponding eigenfunction $u(\ \cdot-t)^{\sfrac q2}$. Its second eigenvalue is $-\Lambda$ with corresponding eigenfunction $\partial_s u(\ \cdot-t)$. These eigenvalues are simple.
	\end{lemma}
	
	\begin{solution}
		These facts are well-known (see, for instance, \cite[4.2.2.~Example: P\"oschl--Teller potentials]{FrLaWe}), but it is easy to give an independent proof. Indeed, \eqref{EL2} says that $u(\ \cdot-t)^{\sfrac q2}$ is a positive $H^2(\R)$-solution of an eigenvalue equation involving the operator. By general Schr\"odinger operator theory, this implies that $u^{\sfrac q 2}$ is the ground state, $-\frac{q^2}{4}\Lambda$ is the lowest eigenvalue, and this eigenvalue is simple. Similarly, \eqref{EL1} says that $\partial_s u(\ \cdot-t)$ is a negative $H^2((t,\infty))\cap H^1_0((t,\infty))$-solution of an eigenvalue equation, so it is the ground state of the Dirichlet realization of the operator on $(t,\infty)$. Since, in one dimension, the eigenvalues of a Schr\"odinger operator with even potential are alternatingly those of the Neumann and the Dirichlet realization, we obtain the assertion about the second eigenvalue.
	\end{solution}
	
	Let us return to the study of equation \eqref{ELQPT}. Note that we have neglected the rank one projector in $\mathcal L_\psi$ when writing \eqref{ELQPT}, which can be justified as follows. When $l\geq 1$, the contribution of the rank one operator vanishes, since $u(\ \cdot-t)$ is a radial (that is, independent of $\omega$) function. When $l=0$, the function $\psi$ is in the kernel of $\mathcal L_\psi$, as we have already observed. When looking for other elements $\phi$ in the kernel, we may thus subtract a suitable multiple of $\psi$ from $\phi$ and are led, in view of \eqref{EL}, to the equation \eqref{ELQPT} without rank one projector.
	
	For $l=0$ we have $\theta_0=-\Lambda$, which, by Lemma \ref{pt}, is the second eigenvalue of the operator. Thus, in this case the $L^2(\R)$-solution space of \eqref{ELQPT} is one-dimensional and spanned by $\partial_s u(\ \cdot-t)$. For $l= 1$ we have $\theta_1= -(d-1+\Lambda)=-\frac{q^2}4 \Lambda$ (as $\Lambda=\Lambda_{FS}$), which, by Lemma \ref{pt}, is the lowest eigenvalue of the operator. Thus, in this case the $L^2(\R)$-solution space of \eqref{ELQPT} is one-dimensional and spanned by $u(\ \cdot-t)^{\sfrac q2}$. Finally, for $l\geq 2$ we have $\theta_l<\theta_1$, and correspondingly there is no non-trival $L^2(\R)$-solution of \eqref{ELQPT}.
	
	Multiplied with a basis of spherical harmonics of the appropriate degree, we see that the kernel of the Hessian is spanned by $\{\psi, \partial_s \psi, \psi^{\sfrac{q}{2}}\omega_i, i=1,\dots, d\}$, as claimed.
	
	\subsection{Proof of Proposition \ref{prop2}}
	\emph{Step 1.} Proposition \ref{prop1} is applicable as $\|u_n\|^2\to C_{a,b}$ and $\|u_n\|_q\to1$ for $n\to\infty$. Passing to a subsequence, we thus obtain the decomposition 
	\[u_n=\lambda_n(u+r_n)(\ \cdot-t_n,\omega)\,,
	\qquad (s,\omega)\in\mathcal C\,,\]
	with the prescribed convergence and orthogonality properties. 
	Defining $\alpha_n\coloneqq\langle r_n, u^{\sfrac{q}{2}}\omega\rangle\in\R^d$ and $\mu_n:=|\alpha_n|$, we can choose an orthogonal matrix $D_n\in\mathcal{O}(d)$ such that $D_n\alpha_n = \mu_n e_d$. It follows that	
	\[ 
	\mu_n \omega_d\circ D_n= \alpha_n\cdot \omega
	\qquad\text{for all}\ \omega\in\Sph^{d-1}\,,
	\]where $\cdot$ denotes the scalar product in $\R^d$. We will abuse the notation slightly by writing $f\circ D_n=f (\ \cdot \ , D_n\ \cdot\ )$ for a function $f$ defined on $\mathcal C$. We define $\tilde R_n\in H^1(\mathcal C)$ to be
	\[\tilde R_n\coloneqq \left(r_n- \alpha_n\cdot\omega u^{\sfrac{q}{2}} \right)\circ D_n^{-1}= r_n\circ D_n^{-1}- \mu_n u^{\sfrac{q}{2}}\omega_d\,, \]
	so $\tilde R_n\circ D_n$ is the remainder term of projecting $r_n$ onto $u^{\sfrac{q}{2}}\omega_i$, $i=1,\dots, d$, in $H^1(\mathcal C)$. As $D_n^{-1}$ only rotates the basis $\{\omega_i\}_{i\in\{1,\dots,d\}}$, the set $\{\omega_i\circ D_n^{-1}\}_{i\in\{1,\dots,d\}}$ spans the spherical harmonics of degree $1$ as well. Therefore, we see that 
	\begin{equation}\label{ortho2}
		\langle \tilde R_n, u^{\sfrac{q}{2}}\omega_i\rangle=0\qquad \text{ for every } i\in\{1,\dots,d\}\,.
	\end{equation} Since $u^{\sfrac{q}{2}}\omega_d$ is orthogonal to $\operatorname{span}\{ u, \partial_s u \}$ in $H^1(\mathcal C)$, we can apply the orthogonality conditions for $r_n$ from Proposition \ref{prop1} to obtain the relations 
	\begin{equation*}
		\langle \tilde R_n, u\rangle=\langle \tilde R_n, \partial_s u\rangle=0\,.
	\end{equation*}
	
	\emph{Step 2.}
	Turning to the convergence properties, as $q>2$, we can expand 
	\begin{align*}
		\left||1+x|^q-1-qx-\frac{q(q-1)}{2}x^2\right|&\lesssim|x|^{q\wedge 3}+|x|^q
		\qquad&&\hspace{-1.8cm}\text{ uniformly in } x\,,
		\\ \left| |1+y|^{\sfrac{2}{q}}-1-\frac{2}{q}y\right|&\lesssim |y|^2 \qquad&&\hspace{-1.8cm}\text{ for } |y|\to 0\,.
	\end{align*}
	Applying the expansions above with $x=r_n u^{-1}$ and 
	\begin{equation*}
		y=\frac{q(q-1)}{2\|u\|^q_q}\int_{\mathcal C}u^{q-2}r_n^2\ \mathrm{d}(s,\omega)+\mathcal O_{n\to\infty}(\|r_n\|^{q\wedge 3}_{q\wedge 3}+\|r_n\|^{q}_{q})
	\end{equation*}
	leads to 
	\begin{align}\notag
		\mathcal \lambda_n^{-2}\|u_n\|^2_q
		&=\left(\int_{\mathcal C}u^q\left|1+\frac{r_n}{u}\right|^q \ \mathrm{d}(s,\omega)\right)^{\sfrac{2}{q}}\notag
		\\&\notag=\|u\|^2_q\left(1+\frac{q(q-1)}{2\|u\|_q^q}\int_{\mathcal C}u^{q-2}r_n^2 \ \mathrm{d}(s,\omega)+\mathcal O_{n\to\infty}(\|r_n\|^{q\wedge 3}_{q\wedge 3}+\|r_n\|^{q}_{q})\right)^{\sfrac{2}{q}}
		\\&\notag=\|u\|^2_q\left(1+\frac{(q-1)}{\|u\|_q^q}\int_{\mathcal C}u^{q-2}r_n^2 \ \mathrm{d}(s,\omega)\right)+\mathcal O_{n\to \infty}(\|r_n\|^{q\wedge 3})\,,
	\end{align}
	where the first order term in the penultimate step vanished due to orthogonality; see \eqref{ortho1}. In the last step, Sobolev embedding and $\|r_n\|\to 0$ for $n\to\infty$ simplified the error $\|r_n\|^{q\wedge 3}_{q\wedge 3}+\|r_n\|^{q}_{q}=\mathcal O_{n\to \infty} (\|r_n\|^{q\wedge 3})$. Using H\"older's inequality, we can verify that the Taylor expansion in use was indeed applicable as $ |y|^2=\mathcal O_{n\to \infty}(\|r_n\|^{q\wedge 3})$.
	Recalling that by orthogonality \[\|u_n\|^2=\lambda_n^2(\|u\|^2+\|r_n\|^2)\,,\] we are able to expand
	$\mathcal F$ to quadratic order:
	\begin{equation}
		\mathcal F(u_n)=\|u_n\|^2-C_{a,b}\|u_n\|_q^2
		=\frac{\lambda_n^2}{2}D^2_u\mathcal F(r_n)+\mathcal O_{n\to \infty} (\|r_n\|^{q\wedge 3})\,.\label{expan}
	\end{equation}
	Above, the terms of order zero vanish due to the optimality of $u$, and the rank one projector in $D^2_u\mathcal F$ can be added as $\langle r_n,u^{q-1}\rangle_2=0$.
	
	Since $\|r_n\|\to 0$ for $n\to\infty$, the expansion \eqref{expan} yields
	\begin{equation}\label{limQ}
		\lim_{n\to\infty}\frac{D^2_u\mathcal F(r_n)}{2\|r_n\|^2}=\lim_{n\to\infty}\frac{\mathcal F(u_n)}{\operatorname{dist}(u_n, \mathcal M)^2}=0\,,
	\end{equation}
	where we used our assumption \eqref{van} in the last step. We recall from Subsection \ref{subsec3.1} that the Hessian of $\mathcal F$ corresponds to an operator $\mathcal L_\psi$ in $L^2(\mathcal C)$. Since its essential spectrum starts at $\Lambda$ (by Weyl's theorem; see, e.g., \cite[Theorem 1.14]{FrLaWe}), we see that $D^2_u\mathcal F$ is positive definite on the $L^2(\mathcal C)$-orthogonal complement of $\operatorname{Ker}(D^2_u\mathcal F)$ and has a spectral gap above $0$, so
	\begin{equation*}
		D^2_u\mathcal F|_{\operatorname{Ker}(D^2_u\mathcal F)^\perp}\geq  \tilde c\|\cdot\|_2^2
	\end{equation*}
	for some $\tilde c>0$. The right side can be improved to the $H^1(\mathcal C)$-norm. To this end, let $\delta>0$ and $\phi\in\operatorname{Ker}(D^2_u\mathcal F)^\perp $. As $\|u\|^{q-2}_{L^{\infty}(\R)}\leq \frac{q}{2}\Lambda$, we can bound
	\begin{equation*}
		D^2_u\mathcal F(\phi)\geq \delta \|\phi\|^2+\left((1-\delta)\tilde c -\delta(q-1)\frac{q}{2}\Lambda\right)\|\phi\|_2^2\geq\delta \|\phi\|^2
	\end{equation*}
	for $\delta$ chosen small enough.
	Therefore, it follows immediately that
	\begin{equation*}
		D^2_u\mathcal F|_{\operatorname{Ker}(D^2_u\mathcal F)^\perp}\asymp\|\cdot\|^2\,,
	\end{equation*}
	that is, the induced norms are equivalent. Moreover, we know that $D^2_u\mathcal F(u^{\sfrac{q}{2}}\omega_d)=0$ by \eqref{KerHess} and $\langle \tilde R_n, u^{\sfrac{q}{2}}\omega_d\rangle=\langle \tilde R_n, u^{\sfrac{q}{2}}\omega_du^{q-2}\rangle_2=0$ by \eqref{ortho2} and \eqref{EL2}. Using this and \eqref{limQ}, we find that
	\begin{equation}\label{conv1}
		\frac{\| \tilde R_n\|^2}{\|r_n\|^2}\asymp\frac{D^2_u\mathcal F( \tilde R_n)}{\|r_n\|^2}=\frac{D^2_u\mathcal F(r_n)}{\|r_n\|^2}\to 0
	\end{equation}
	for $n\to\infty$. The orthogonality relations \eqref{ortho2} and the equation \eqref{EL2} for $u^{\sfrac{q}{2}}$ along with $\|\omega_d\|^2_{L^2(\mathbb S^{d-1})}=|\mathbb S^{d-1}| d^{-1}$ imply
	\begin{equation*}
		\| \tilde R_n\|^2+\frac{q-1}{d}\|u^{q-1}\|_2^2 \mu_n^2 =\|r_n\|^2\,.
	\end{equation*}
	This and the asymptotics \eqref{conv1} show that $\mu_n^2\|r_n\|^{-2} \to d((q-1) \|u^{q-1}\|^2_2)^{-1}$. It follows that $\mu_n=\mathcal O_{ n\to \infty}(\|r_n\|)$, and, unless $r_n=0$, we have $\mu_n\neq 0$ for all sufficiently large $n$. We finally set $R_n:=\mu_n^{-1}\tilde R_n$ when $\mu_n\neq 0$ (and $R_n=0$ when $\mu_n=0$). Then $\|R_n\|=o_{n\to \infty}(1)$, and the above orthogonality conditions for $\tilde R_n$ translate into orthogonality conditions for $R_n$.
	\qed
	
	\section{Non-vanishing of the quartic order}\label{sec4}
	
	\subsection{Quartic expansion of the deficit functional}
	
	As the denominator in \eqref{finalest} equals $\lambda_n^4 \|r_n\|^4$, which is comparable to $\mu_n^4$ for $n\to\infty$, we aim to expand the numerator to fourth order in $|\mu_n|$ and expect lower order terms to vanish. In fact, the decomposition \eqref{dec3} formally leads to a quartic expansion of the functional $\mathcal F$. However, we cannot control arbitrary perturbations $R_n$ in terms of $\mu_n$ yet. We only know that $\|R_n\|\to0$ for $n\to\infty$. In order to conduct perturbation theory, it would be more beneficial to have control of the $L^\infty$-norm of $R_n$ as suggested in \cite{Fr}. In general, an expansion of the $L^{q}$-norm up to fourth order requires $q\geq 4$. However, we will circumvent this issue by splitting the domain of integration in the $L^q$-norm: The remainder term $|R_n|$ is cut off by $u^{\sfrac{q}{2}}$, which stems from the non-trivial zero mode, allowing expansions to arbitrary order. This approach for $1$ instead of $u^{\sfrac{q}{2}}$ simplifies the computations for the quartic expansion in \cite{Fr}, and we expect it to be applicable to prove other degenerate stability results in the future.

	\begin{lemma}[Quartic order expansion of $\mathcal F$]\label{lem1}
		If $(u_n)_n$ is as in Proposition \ref{prop2}, we have in the notation of that proposition,
		\begin{equation*}
			\lambda_n^{-2}\mathcal F(u_n)=\lambda_n^{-2}(\|u_n\|^2-C_{a,b}\|u_n\|^2_q)
			\geq(A)+(B)+\mathcal O_{n\to \infty}(|\mu_n|^5)\,,
		\end{equation*}
		where for some $n$-independent constant $C<\infty$ we define
		\begin{align*}
			(A)&\coloneqq \mu_n^2\left((1-C|\mu_n|^{(q-2)\wedge 1})\|R_n\|^2-(q-1)\int_{\mathcal{C}}(u^{q-2}R_n^2+(q-2)\mu_n u^{2q-3}\omega_d^2R_n)\ \mathrm{d}(s,\omega)\right),\\
			(B)&\coloneqq\mu_n^4\frac{(q-1)(q-2)}{4}\left(\frac{q-1}{\|u\|_q^q}\left(\int_{\mathcal{C}}u^{2q-2}\omega_d^2\ \mathrm{d}(s,\omega)\right)^2-\frac{q-3}{3}\int_{\mathcal{C}}u^{3q-4}\omega_d^4\ \mathrm{d}(s,\omega)\right).
		\end{align*}
	\end{lemma}
	
	It is instructive to consider this result in view of the secondary non-degeneracy condition in Definition \ref{def:secnondeg}. Setting $\phi=\mu_n R_n$ and $g=\mu_n u^{\sfrac{q}{2}}\omega_d$, $(A)$ corresponds (to leading order in $\mu_n$) to the bilinear and linear part   in $\phi$ of \eqref{secnondeg1}, that is, $2^{-1}(D_u^2\mathcal F(\phi,\phi)+D_u^3\mathcal F(g,g,\phi))$, while $(B)$ is the constant part $24^{-1}D_u^4\mathcal F(g,g,g,g)$. Hence, the main goal of the next subsection is to bound them together from below by a positive constant.
	
	\begin{solution}
		As the prerequisites for Proposition \ref{prop2} are satisfied, the decomposition \eqref{dec3} is available.
		Exploiting the translation and rotation invariance of the assumptions and the expansion in the theorem, we may assume without loss of generality that
		\begin{equation*}
			u_n=\lambda_n (u+\mu_n(u^{\sfrac{q}{2}}\omega_d+R_n))\,.
		\end{equation*}
		We consider the set of points in $\mathcal C$ where $|R_n|<u^{\sfrac{q}{2}}$ and the one where $|R_n|\geq u^{\sfrac{q}{2}}$, separately.
		
		\emph{Step 1.}
		In the set where $|R_n|<u^{\sfrac{q}{2}}$ we have \[|\mu_n||u^{\sfrac{q}{2}}\omega_d+R_n|u^{-1}\leq2|\mu_n|u^{\sfrac{q}{2}-1}\leq 2|\mu_n|\left(\frac{q}{2}\Lambda \right)^{\sfrac{1}{2}}\leq \frac{1}{2}\]
		for $n$ large enough. Then we can write 
		\begin{equation}\label{factout}
			|\lambda_n|^{-q}|u_n|^q= u^q|1+\mu_n(u^{\sfrac{q}{2}}\omega_d+R_n)u^{-1}|^q
		\end{equation} 
		and expand the last term around $1$.
		Applying the $L^q(\{|R_n|<u^{\sfrac{q}{2}}\})$-norm to \eqref{factout}, the order of the error in $|\mu_n|$ of the expansion on the right side is preserved since $u^{q}$ is integrable. Thus, we may expand the $L^q(\{|R_n|<u^{\sfrac{q}{2}}\})$-norm to arbitrary order. We will only need an expansion to quartic order:
		\begin{align}
			|\lambda_n|^{-q}& \int_{\{|R_n|<u^{\sfrac{q}{2}}\}}|u_n|^q\ \mathrm{ d}(s,\omega) =\int_{\{|R_n|<u^{\sfrac{q}{2}}\}}u^q \bigg(1+q\mu_n(u^{\sfrac{q}{2}}\omega_d+R_n)u^{-1}\notag\\&
			+\frac{1}{2}q(q-1)\mu^2_n(u^{\sfrac{q}{2}}\omega_d+R_n)^2u^{-2}+\frac{1}{6}q(q-1)(q-2)\mu^3_n ((u^{\sfrac{q}{2}}\omega_d)^3 +3(u^{\sfrac{q}{2}}\omega_d)^2 R_n)u^{-3}\notag	\\&+\frac{1}{24}q(q-1)(q-2)(q-3)\mu^4_n(u^{\sfrac{q}{2}}\omega_d)^4u^{-4}\bigg)\ \mathrm{ d}(s,\omega)+ \mathcal O_{n\to \infty}(|\mu_n|^5+ |\mu_n|^3 \|R_n\|^2)\,,\label{quartexp}
		\end{align}
		where we absorbed some of the third and fourth order terms into the error. In particular, we used
		\[\left|\mu_n^4\int_{\{|R_n|<u^{\sfrac{q}{2}}\}} (u^{\sfrac{q}{2}}\omega_d)^3 R_n u^{q-4}\ \mathrm{d}(s,\omega)\right|\lesssim |\mu_n|^4\|R_n\|_2\lesssim|\mu_n|^5+ |\mu_n|^3 \|R_n\|^2\,.\]
		
		\emph{Step 2.} We can surely expand the $L^q$-norm to second order for $q>2$ on $\{|R_n|\geq u^{\sfrac{q}{2}}\}$:
		\begin{align}
			|\lambda_n|^{-q} &\int_{\{|R_n|\geq u^{\sfrac{q}{2}}\}}|u_n|^q\ \mathrm{ d}(s,\omega)\notag \\&=\int_{\{|R_n|\geq u^{\sfrac{q}{2}}\}} \bigg(u^q+q\mu_n(u^{\sfrac{q}{2}}\omega_d+R_n)u^{q-1}+\frac{1}{2}q(q-1)\mu^2_n(u^{\sfrac{q}{2}}\omega_d+R_n)^2u^{q-2}
			\bigg)\mathrm{ d}(s,\omega)\notag\\&+ \mathcal O_{n\to \infty}\left( \int_{\{|R_n|\geq u^{\sfrac{q}{2}}\}}\left(|\mu_n|^{3\wedge q}  |u^{\sfrac{q}{2}}\omega_d+R_n|^{3\wedge q}+|\mu_n|^{ q} |u^{\sfrac{q}{2}}\omega_d+R_n|^{ q}\right)\ \mathrm{ d} (s,\omega)\right),\label{quadexp}
		\end{align}
		where we used that $u^{q-(q\wedge3)}\lesssim 1$ uniformly in $s$ in the argument of $\mathcal O_{n\to \infty}(\ \cdot\ )$. Denoting $p\in\{q,3\wedge q\}\subset (2, 2^*)$, we realize that
		\begin{equation}
			\int_{\{|R_n|\geq u^{\sfrac{q}{2}}\}}|\mu_n|^{ p} |u^{\sfrac{q}{2}}\omega_d+R_n|^{ p}\ \mathrm{ d} (s,\omega)\lesssim 	|\mu_n|^{ p} \|R_n\|_p^p \lesssim 	|\mu_n|^{ p} \|R_n\|^p\,. \label{est1}
		\end{equation}
		As $u^{\sfrac{q}{2}}\leq |R_n| $ holds pointwise and $(u^{\sfrac{q}{2}-1})^m$ is integrable for every $m>0$, we obtain 
		\begin{align}
			\frac{1}{6}q(q-1)(q-2)\mu_n^3	&\int_{\{|R_n|\geq u^{\sfrac{q}{2}}\}} \!\! u^{q-3}(u^{\sfrac{q}{2}}\omega_d)^2 \left((u^{\sfrac{q}{2}}\omega_d)+3 R_n	+\frac{1}{4}(q-3)\mu_n(u^{\sfrac{q}{2}}\omega_d)^2u^{-1}\right) \mathrm{ d}(s,\omega)\notag \\
			&=
			\mathcal O_{n\to \infty}(|\mu_n|^3 \|R_n\|^2)\,.\label{est2}
		\end{align}
		Inserting \eqref{est1} and \eqref{est2} into \eqref{quadexp} gives us a quartic expansion as in \eqref{quartexp} but over $\{|R_n|\geq u^{\sfrac{q}{2}}\}$ and with an additional error $\mathcal O_{n\to \infty}(\|\mu_nR_n\|_q^q+\|\mu_nR_n\|_{q\wedge3}^{q\wedge3})$. 
		
		\emph{Step 3.} By simple manipulations, we may reduce the overall error for the quartic expansion of the $L^q(\mathcal C)$-norm to $\mathcal O_{n\to \infty}(|\mu_n|^5+|\mu_n|^{q\wedge3}\|R_n\|^2)$. Using this expansion, we compute
		\begin{align*}
			\lambda_n^{-2}\|u_n\|^2_q=\|u\|_q^2 \bigg(1&+\|u\|_q^{-q}\int_{\mathcal C}\bigg( \! (q-1)\mu^2_n(u^{q}\omega^2_d+R^2_n)u^{q-2}
			\! +(q-1)(q-2)\mu^3_n(u^{\sfrac{q}{2}}\omega_d)^2 R_nu^{q-3}\notag	\\&+\frac{1}{12}(q-1)(q-2)(q-3)\mu^4_n(u^{\sfrac{q}{2}}\omega_d)^4u^{q-4}\bigg)\ \mathrm{ d}(s,\omega)\\&-\frac{1}{4}(q-2)\|u\|_q^{-2q}(q-1)^2\mu^4_n\left(\int_{\mathcal C}u^{q}\omega_d^2u^{q-2}\notag\ \mathrm{ d}(s,\omega)\right)^2\bigg)\\&+ \mathcal O_{n\to \infty}(|\mu_n|^5+ |\mu_n|^{q\wedge3} \|R_n\|^2)\,,
		\end{align*}
		where the whole first order term, the mixed term of second order, and the term including $\omega_d^3$ of third order vanish due to orthogonality relations of the spherical harmonic $\omega_d$ and $R_n$; see Proposition \ref{prop2}.
		Together with the expansion of the $H^1(\mathcal C)$-norm,
		\[	\lambda_n^{-2}\|u_n\|^2=\|u\|^2+\mu_n^2(\|u^{\sfrac{q}{2}}\omega_d\|^2+\|R_n\|^2)\,,\]
		we obtain the desired expansion of the functional $\mathcal F$ to quartic order. Note that the terms of order zero vanish by optimality of $u$, and the $R_n$-independent second order terms cancel due to equation \eqref{EL2}. The stated lower bound now follows from replacing $\mathcal O_{n\to \infty}(|\mu_n|^{q\wedge3} \|R_n\|^2)$ by a lower bound of the form $-C|\mu_n|^{q\wedge3} \|R_n\|^2$, $C<\infty$.
	\end{solution}
	
	%%%%%%%%%%%%%%%%%%%%%%%%%
	
	\subsection{Bounding $(A) + (B)$ from below}
	
	In order to prove Proposition \ref{prop3}, we are going to estimate $(A) + (B)$ from Lemma \ref{lem1} from below by the leading order $\mu_n^4$ times a positive $(q,d)$-dependent constant. To this end, we will use two well-known identities frequently hereafter:
	\begin{equation}\label{ident}
		\int_{\mathbb S^{d-1}}\omega_d^{n}\ \mathrm{ d} \omega= |\mathbb S^{d-1}|\frac{\Gamma\left(\frac{d}{2}\right)\Gamma\left(\frac{n+1}{2}\right)}{\Gamma\left(\frac{1}{2}\right)\Gamma\left(\frac{d+n}{2}\right)}\mathbbm 1_{\{n\in 2\N_0\}}\text{ and }\int_\R \frac{|\sinh(s)|^n}{\cosh^\nu(s)}\ \mathrm{ d} s= \frac{\Gamma\left(\frac{\nu-n}{2}\right)\Gamma\left(\frac{n+1}{2}\right)}{\Gamma\left(\frac{\nu+1}{2}\right)}
	\end{equation}
	for $\nu>n$, $n\in \N_0$. After passing to spherical coordinates, the first integral follows from \cite[3.621, Eq.~5]{GrRy}, while the second one is a direct consequence of \cite[3.512, Eq.~2]{GrRy}.
	
	Let us start with $(B)$. It is independent of $R_n$ and can thus be computed explicitly by means of the identities \eqref{ident}. We find that
	\begin{equation}
		(B)=\mu_n^4\frac{\beta^{3q-4}}{\alpha}\frac{|\mathbb S^{d-1}|}{4d^2}\frac{\Gamma\left(\frac{3q-4}{q-2}\right)\sqrt{\pi}}{\Gamma\left(\frac{3q-4}{q-2}+\frac{1}{2}\right)} (q-1)(q-2)  \left(\frac{q(5q-6)}{2(3q-2)}-\frac{d(q-3)}{(d+2)}\right).\label{Bcomp}
	\end{equation}
	
	In order to estimate $(A)$, we will expand $R_n$ in spherical harmonics,
	\begin{equation}\label{sphdec}
		R_n(s,\omega)=\sum_{l,m} a_{l,m}(s) Y_{l,m}(\omega)\,,
	\end{equation}
	where $a_{l,m}$ are $s$-dependent coefficients and $(Y_{l,m})_m$ is an $L^2(\mathbb S^{d-1})$-orthonormal basis of spherical harmonics of degree $l$, which we introduced in Section \ref{sec3}. We will choose \begin{equation}\label{Y}
		Y_{0,0}\coloneqq\frac{1}{\sqrt{|\mathbb S^{d-1}|}}\qquad \text{ and }\qquad
		Y_{2,0}(\omega_d)\coloneqq\sqrt{\frac{d^2(d+2)}{2(d-1)|\mathbb S^{d-1}|}} \left(\omega_d^2-\frac{1}{d}\right),
	\end{equation} where the normalizing constants can be computed using the first identity from \eqref{ident}. The $L^2(\mathbb S^{d-1})$-orthogonality of the spherical harmonics $Y_{l,m}$ allows us to state the $L^2(\mathcal C)$-orthogona\-lity conditions from Proposition \ref{prop2} (see \eqref{eq:prop1orthol2}) in terms of the coefficients:
	\begin{equation}\label{ortho3}
		\langle a_{0,0}, u^{q-1} \rangle_{L^2(\R)}=\langle a_{0,0}, (q-1) u^{q-2}\partial_s u \rangle_{L^2(\R)}=\langle a_{1,m}, (q-1) u^{q-2}u^{\sfrac{q}{2}} \rangle_{L^2(\R)}=0\,.
	\end{equation} 
	
	There will be two lemmas that provide sharp lower bounds on $(A)$. The first estimate says that, except for $l=0,m=0$ and $l=2,m=0$, the contribution of the coefficients is bounded from below by $0$. The other one takes care of the missing cases by a `completing the square'-argument.
	
	\begin{lemma}[Negligible energies]\label{lem2}
		Inserting \eqref{sphdec} for $R_n$, we have for any $n\in\N$,
		\begin{equation}\label{pos}
			(A)\geq \mu_n^4\sum_{l\in\{0,2\}}\mathcal E^{(l)}\left(\frac{a_{l,0}}{\mu_n}\right),
		\end{equation} 
		where, for $l\in\{0,2\}$ and $g\in H^1(\R)$,
		\[\mathcal E^{(l)}(g)\coloneqq \int_\R ((1-C|\mu_n|^{(q-2)\wedge1})((\partial_s g)^2+(2d\mathbbm1_{\{l=2\}}(l)+\Lambda)g^2)-(q-1)u^{q-2}g^2 -2 f^{(l)}g) \ \mathrm{d}s\]
		with
		\[ f^{(l)}\coloneqq M u^{2q-3} \left(\sqrt{\frac{2(d-1)}{d+2}}\mathbbm1_{\{l=2\}}(l)+\mathbbm1_{\{l=0\}}(l)\right),
		\qquad
		M\coloneqq(q-1)(q-2) \frac{\sqrt{|\mathbb S^{d-1}|}}{2d}\,.
		\]
	\end{lemma}
	
	Here to lighten the notation, we do not reflect the dependence of $\mathcal E^{(l)}$ on $n$. Since $\mu_n\to 0$, this dependence will be weak.
	
	\begin{solution}
		By means of \eqref{Y}, we expand $\omega_d^2$ in spherical harmonics:
		\[\omega_d^2=(\omega_d^2-d^{-1})+d^{-1}=\sqrt{\frac{2(d-1)|\mathbb S^{d-1}|}{d^2(d+2)}}Y_{2,0}+\frac{\sqrt{|\mathbb S^{d-1}|}}{d}Y_{0,0}\,.\]
		As a result, expanding $R_n$ in spherical harmonics as well, we can write $(A)$ as 
		\begin{align}
			\notag\mu_n^2&\left((1-C |\mu_n|^{(q-2)\wedge1})\|R_n\|^2-(q-1)\int_{\mathcal{C}}\left(u^{q-2}R_n^2\ +(q-2)\mu_n u^{2q-3}\omega_d^2R_n\right)\, \mathrm{d}(s,\omega)\right)\\&\notag
			=\mu_n^2\int_\R \Bigg( \sum_{l,m} \Bigg( (1-C |\mu_n|^{(q-2)\wedge1})((\partial_s a_{l,m})^2+(l(l+d-2)+\Lambda)a_{l,m}^2)-(q-1)u^{q-2} a_{l,m}^2 \Bigg) \\&
			\qquad\qquad -(q-1)(q-2)\mu_n \frac{\sqrt{|\mathbb S^{d-1}|}}{d} \left(u^{2q-3}\sqrt{\frac{2(d-1)}{d+2}} a_{2,0}+ u^{2q-3} a_{0,0}\right)\Bigg) \, \mathrm{d}s \,. \label{sphbound}
		\end{align}
		The terms with $(l,m)\in\{(0,0),(2,0)\}$ are the terms that appear on the right side in the lemma. We will now show that the remaining terms are bounded from below by zero. This will imply the lemma.
		
		According to Lemma \ref{pt}, the lowest eigenvalue of the operator $-\partial^2_s -(q-1)u^{q-2}$ in $L^2(\R)$ is $-\frac{q^2}{4}\Lambda$. This, together with the bound $\|u\|^{q-2}_{L^\infty(\R)}\leq \frac{q}{2}\Lambda$, implies that, once $n$ is so large that $C|\mu_n|^{(q-2)\wedge1}\leq 1$,
		\begin{align*}
			& (1-C |\mu_n|^{(q-2)\wedge1}) (-\partial^2_s +l(l+d-2)+\Lambda)-(q-1)u^{q-2} \\
			& \quad \geq (1-C|\mu_n|^{(q-2)\wedge1}) \left( -\partial^2_s +l(l+d-2)+\Lambda -(q-1)u^{q-2} \right) - C|\mu_n|^{(q-2)\wedge1}(q-1)\frac{q}{2}\Lambda \\
			& \quad \geq (1-C|\mu_n|^{(q-2)\wedge1}) \left( - \frac{q^2}{4}\Lambda + l(l+d-2) + \Lambda \right) -C|\mu_n|^{(q-2)\wedge1}(q-1)\frac{q}{2}\Lambda \,.
		\end{align*}
		We have $- \frac{q^2}{4}\Lambda + \Lambda = - \frac{q^2-4}{4} \Lambda_{FS} = - (d-1)$. Since for $l\geq 2$ we have $l(l+d-2)\geq 2d>d-1$, it follows that, if $n$ is large enough, then for all $l\geq 2$ we have
		\begin{equation}\label{lowerbdd}
			(1-C |\mu_n|^{(q-2)\wedge1}) (-\partial^2_s +l(l+d-2)+\Lambda)-(q-1)u^{q-2} \geq 0 \,.
		\end{equation}
		Therefore, we can bound the terms in the sum over $l,m$ in \eqref{sphbound} with $l\geq 2$ by $0$ from below; however, we will keep the summand with $(l,m)=(2,0)$. 
		
		We now bound the term with $l=1$. We set $Z_m:=\langle u^{\sfrac q2},a_{1,m}\rangle_{L^{2}(\R)} \|u\|_{L^{q}(\R)}^{-q}$ and write
		\begin{align*}
			& (1-C|\mu_n|^{(q-2)\wedge1})(\| \partial_s a_{1,m}\|^2_{L^2(\R)}+(d-1+\Lambda)\| a_{1,m}\|^2_{L^2(\R)})-(q-1)\langle a_{1,m}, u^{q-2} a_{1,m}\rangle_{L^2(\R)}\\
			& \quad =	(1-C|\mu_n|^{(q-2)\wedge1}) \Big( \| \partial_s (a_{1,m}-Z_m u^{\sfrac{q}{2}})\|^2_{L^2(\R)}+(d-1+\Lambda)\| a_{1,m}-Z_m u^{\sfrac{q}{2}}\|^2_{L^2(\R)} \\
			&\qquad\qquad\qquad\qquad\qquad\  -(q-1)\langle (a_{1,m}-Z_m u^{\sfrac{q}{2}}), u^{q-2} (a_{1,m}-Z_m u^{\sfrac{q}{2}})\rangle_{L^2(\R)} \Big) \\
			& \qquad - C |\mu_n|^{(q-2)\wedge1}(q-1) \left( \langle (a_{1,m}- Z_m u^{\sfrac{q}{2}}), u^{q-2} (a_{1,m} - Z_m u^{\sfrac{q}{2}}) \rangle_{L^2(\R)}-Z_m^2\|u^{q-1}\|^2_{L^2(\R)} \right).
		\end{align*}
		Here we used the equation \eqref{EL2} for $u^{\sfrac{q}{2}}$, the third condition in \eqref{ortho3} and the assumption $\Lambda=\Lambda_{FS}$. Since $a_{1,m}-Z_m u^{\sfrac{q}{2}}$ is $L^2(\R)$-orthogonal to $u^{\sfrac{q}{2}}$, which is the ground state of the operator $-\partial^2_s -(q-1)u^{q-2}$ in $L^2(\R)$, and since the second eigenvalue of this operator is $-\Lambda$, we can bound
		\begin{align*}
			& \| \partial_s (a_{1,m}-Z_m u^{\sfrac{q}{2}})\|^2_{L^2(\R)}+(d-1+\Lambda)\| a_{1,m}-Z_m u^{\sfrac{q}{2}}\|^2_{L^2(\R)} \\
			& \qquad\qquad -(q-1)\langle (a_{1,m}-Z_m u^{\sfrac{q}{2}}), u^{q-2} (a_{1,m}-Z_m u^{\sfrac{q}{2}})\rangle_{L^2(\R)} \\
			& \quad \geq (d-1) \| a_{1,m}-Z_m u^{\sfrac{q}{2}}\|^2_{L^2(\R)} \,.
		\end{align*}
		This, together with $\|u\|^{q-2}_{L^\infty(\R)}\leq \frac{q}{2}\Lambda$, implies that, when $C|\mu_n|^{(q-2)\wedge1}\leq 1$,
		\begin{align*}
			& (1-C|\mu_n|^{(q-2)\wedge1})(\| \partial_s a_{1,m}\|^2_{L^2(\R)}+(d-1+\Lambda)\| a_{1,m}\|^2_{L^2(\R)})-(q-1)\langle a_{1,m}, u^{q-2} a_{1,m}\rangle_{L^2(\R)}\\
			& \quad \geq \left( (1-C|\mu_n|^{(q-2)\wedge1}) (d-1) - C |\mu_n|^{(q-2)\wedge1}(q-1)\frac{q}{2}\Lambda \right) \| a_{1,m}-Z_m u^{\sfrac{q}{2}}\|^2_{L^2(\R)}\,.
		\end{align*}
		This is non-negative for all sufficiently large $n$. Thus, we conclude that \eqref{pos} holds.
	\end{solution}
	
	Lemma \ref{lem2} motivates to study the minimization problems
	\begin{align*}
		E^{(0)} & \coloneqq\inf \left\{ \mathcal E^{(0)}(g) :\ g\in H^1(\R), \langle u^{q-1}, g\rangle_{L^2(\R)}=\langle u^{q-2}\partial_s u, g\rangle_{L^2(\R)} = 0 \right\}, \\
		E^{(2)} & \coloneqq\inf \left\{ \mathcal E^{(2)}(g) :\ g\in H^1(\R) \right\}.
	\end{align*}
	The energy functionals $\mathcal E^{(l)}$ are of the form quadratic plus linear, and therefore the corresponding minimization problems $E^{(l)}$ are abstractly solvable by a `completion of the square'-argument. Producing concrete numerical values, however, is not straightforward, in particular for $l=2$. These numerical values are necessary in order to verify the secondary non-degeneracy condition. We stress that this difficulty is already present in the model case where $\mu_n$ is replaced by zero. To avoid distraction from the main idea of the proof, we state here the outcome of the `completion of the square'-argument and provide a sketch of the argument but defer the details for $l=2$ to the next section.
	
	\begin{lemma}[Energy of the degree $0$ solution]\label{lem5new}
		As $n\to\infty$, we have
		\[ E^{(0)} = - \frac{\beta^{3q-4}}{\alpha}\frac{|\mathbb S^{d-1}|}{4d^2}\frac{\Gamma \left(\frac{3q-4}{q-2}\right)\sqrt \pi}{\Gamma \left(\frac{3q-4}{q-2}+\frac{1}{2}\right)}\frac{q(q-2)^3}{4(3q-2)}+ \mathcal O_{n\to\infty}(|\mu_n|^{(q-2)\wedge1})\,.\]
	\end{lemma}
	
	\begin{lemma}[Energy of the degree $2$ solution]\label{lem8new}
		As $n\to\infty$, we have
		\begin{align*}
			E^{(2)} & = -\frac{\beta^{3q-4}}{\alpha}\frac{|\mathbb S^{d-1}|}{4d^2}\frac{\Gamma \left(\frac{3q-4}{q-2}\right)\sqrt \pi}{\Gamma \left(\frac{3q-4}{q-2}+\frac{1}{2}\right)}\frac{q(q-1)(q-2)(d-1)}{(d+2)P(-1)}
			\sum_{k=0}^\infty \left(P(k-\xi)- P(k)\right) \\
			& \quad + \mathcal O_{n\to \infty}(|\mu_n|^{(q-2)\wedge1})
		\end{align*}
		with a smooth function $P:[-1,\infty)\to (0,\infty)$ defined by
		\begin{equation}
			\label{eq:defp}
			P(x)\coloneqq \frac{\Gamma\left(x+\frac{3}{2}\right)\Gamma\left(x+2\mathfrak b-1\right)\Gamma\left(x+2\mathfrak b\right)}{\Gamma\left(x+\mathfrak b-\mathfrak a+1\right)\Gamma\left(x+\mathfrak b+\mathfrak a  +1\right)\Gamma\left(x+2\mathfrak b+\frac{1}{2}\right)}\,,\qquad x\geq -1\,,
		\end{equation} and 
		\begin{equation}
			\label{xiab}
			\mathfrak a \coloneqq\frac{\sqrt{1+\frac{2d}{\Lambda}}}{q-2} \,,
			\qquad
			\mathfrak b \coloneqq \frac{2q-3}{q-2}\,,\qquad\xi\coloneqq\mathfrak b-\mathfrak a\,.
		\end{equation}
	\end{lemma}
	For the proof of these two lemmas, we consider for $l\in\{0,2\}$ the one-dimensional Schr\"odinger operators 
	\begin{equation}
		\label{eq:defhnl}
		h_n^{(l)}\coloneqq(1-C|\mu_n|^{(q-2)\wedge1})(-\partial_s^2+2d\mathbbm1_{\{l=2\}}(l)+\Lambda)-(q-1)u^{q-2} \,.
	\end{equation}
	They can be considered as self-adjoint, lower semibounded operators in the Hilbert space $L^2(\R)$ with form domain $H^1(\R)$ and operator domain $H^2(\R)$.
	
	\begin{proof}[Proof of Lemma \ref{lem8new}]
		As we have seen in \eqref{lowerbdd} in the proof of Lemma \ref{lem2}, the operator $h_n^{(2)}$ is bounded from below by $d+1+ \mathcal O(|\mu_n|^{(q-2)\wedge 1})$. Thus, for all sufficiently large $n$ the operator is boundedly invertible. We can write for any $g\in H^1(\R)$,
		\begin{align*}
			\mathcal E^{(2)}(g) & = \| (h_n^{(2)})^{1/2} g - (h_n^{(2)})^{-1/2} f^{(2)} \|_{L^2(\R)}^2 - \langle f^{(2)}, (h_n^{(2)})^{-1} f^{(2)} \rangle_{L^2(\R)}\,.
		\end{align*}
		Thus, $E^{(2)} = \inf_g \mathcal E^{(2)}(g) = - \langle f^{(2)}, (h_n^{(2)})^{-1} f^{(2)} \rangle_{L^2(\R)}$, and the infimum is attained at the unique $g$ that satisfies $(h_n^{(2)})^{1/2} g = (h_n^{(2)})^{-1/2} f^{(2)}$. The latter is equivalent to having $g\in H^2(\R)$ and $h_n^{(2)} g = f^{(2)}$. We have not been able to find a simple explicit expression of this solution $g$ (not even for $\mu_n=0$), but we have been able to find a power series representation of it, which allows us to deduce the formula stated in the lemma. We defer the details of this rather lengthy analysis to the next section.
	\end{proof}
	
	The idea of the proof of Lemma \ref{lem5new} is similar to that of Lemma \ref{lem8new}, and in this case, in fact, an explicit expression for the solution is available. There is a different complication, however, namely, in the first part of the proof. This comes from the fact that the operator $h_n^{(0)}$ is not positive definite. Indeed, from Lemma \ref{pt} we know that at $\mu_n=0$ its lowest two eigenvalues are $-(\frac{q^2}4-1)\Lambda$ and $0$. We need to remove these two unstable directions in order to obtain a boundedly invertible operator. As we will show, this is achieved by the orthogonality conditions in the definition of $E^{(0)}$. This is not completely obvious since the functions in these orthogonality conditions are not eigenfunctions of the operator.
	
	\begin{proof}[Proof of Lemma \ref{lem5new}]
		\emph{Step 1.} We denote by $\Pi$ the $L^2(\R)$-orthogonal projection onto the $L^2(\R)$-orthogonal complement of $\operatorname{span}\{u^{q-1},u^{q-2}\partial_s u\}$. We claim that for all sufficiently large $n$ the operator $\Pi h_n^{(0)}\Pi$, considered in the Hilbert space $\ran\Pi$, is bounded from below by a positive constant and, consequently, boundedly invertible on that space. Since the bottom of the spectrum of $\Pi h_n^{(0)}\Pi$ depends continuously on $\mu_n$, it suffices to prove this assertion for the operator $h^{(0)}$, defined similarly as $h_n^{(0)}$ but with $\mu_n$ set to zero. 
		
		By the computations in Subsection \ref{subsec3.1}, we know that the Hessian of $\mathcal F$ at $u$, restricted to radial functions, is given by the operator $h^{(0)} + (q-2)\|u\|_q^{-q} |u^{q-1}\rangle\langle u^{q-1}|$ in $L^2(\R)$. Since the Hessian is positive semidefinite and since
		$$
		\Pi h^{(0)} \Pi = \Pi \left( h^{(0)} + (q-2)\|u\|_q^{-q} |u^{q-1}\rangle\langle u^{q-1}| \right) \Pi \,,
		$$
		we deduce that $\Pi h^{(0)}\Pi$ is positive semidefinite. Since the essential spectrum of this operator starts at $\Lambda>0$ (as before by Weyl's theorem; see, e.g., \cite[Theorem 1.14]{FrLaWe}), it suffices to prove that $0$ is not an eigenvalue of $\Pi h^{(0)}\Pi$. Thus, assume that $\Pi h^{(0)}\Pi\phi =0$ for some $\phi\in H^2(\R)$. Then the above computation shows that $\Pi\phi$ lies in the kernel of the Hessian of $\mathcal F$ at $u$, restricted to radial functions. Hence, by \eqref{KerHess}, $\Pi\phi= c_1 u + c_2 \partial_s u$ for some constants $c_1,c_2\in\R$. Since even and odd functions are mutually orthogonal, we deduce that
		$$
		0 = \langle u^{q-1},\Pi\phi\rangle_{L^2(\R)} = c_1 \int_\R u^{q} \ \mathrm{d}s
		\qquad\text{and}\qquad
		0 = \langle u^{q-2},\partial_s \Pi\phi\rangle_{L^2(\R)} = c_2 \int_\R u^{q-2} \partial^2_s u \ \mathrm{d}s \,.
		$$
		Thus, $c_1=c_2=0$ and $\Pi\phi=0$. This means that $\Pi h^{(0)}\Pi$ has trivial kernel on the space $\ran\Pi$, as claimed.
		
		\emph{Step 2.} Now we can proceed similarly as in the proof of Lemma \ref{lem8new}. We can write for any $g\in H^1(\R)\cap\ran\Pi$,
		$$
		\mathcal E^{(0)}(g) = \| (\Pi h_n^{(0)}\Pi)^{1/2} g - (\Pi h_n^{(0)}\Pi)^{-1/2} \Pi f^{(0)} \|_{\ran\Pi}^2 - \langle  \Pi f^{(0)}, (\Pi h_n^{(0)}\Pi)^{-1}  \Pi f^{(0)} \rangle_{\ran\Pi} \,,
		$$
		where we consider the natural inner product and norm that $\ran\Pi$ inherits from $L^2(\R)$. We conclude that $E^{(0)} = \inf_{g\in H^1(\R)\cap\ran\Pi} \mathcal E^{(0)}(g) = - \langle  \Pi f^{(0)}, (\Pi h_n^{(0)}\Pi)^{-1} \Pi  f^{(0)} \rangle_{\ran\Pi}$, and the infimum is attained at the unique $g$ that satisfies $(\Pi h_n^{(0)}\Pi)^{1/2} g = (\Pi h_n^{(0)}\Pi)^{-1/2}  \Pi f^{(0)}$. The latter is equivalent to having $g\in H^2(\R)\cap\ran\Pi$ and $\Pi h_n^{(0)}\Pi g =  \Pi f^{(0)}$. This means that
		\begin{equation}
			\label{eq:energy0proof}
			h_n^{(0)}\Pi g = f^{(0)} + c_3 u^{q-1} + c_4 u^{q-2}\partial_s u
		\end{equation}
		for some constants $c_3,c_4\in\R$.
		
		Exploiting \eqref{EL}, we compute directly
		\begin{align*}
			h_n^{(0)}u&=-(q-2+C|\mu_n|^{(q-2)\wedge 1})u^{q-1}\,,\\
			h_n^{(0)}u^{q-1}&=-(1-C|\mu_n|^{(q-2)\wedge 1})q(q-2)\Lambda u^{q-1}+\frac{(q-1)}{q}(2(q-2)-(3q-4)C|\mu_n|^{(q-2)\wedge 1})u^{2q-3}\,,
		\end{align*}
		and therefore, if we set $$g=K u^{q-1}+Lu$$ with constants $K,L\in\R$ to be determined, then
		\begin{align*}
			h_n^{(0)}g & = \left( -(1-C|\mu_n|^{(q-2)\wedge 1})q(q-2)\Lambda K - (q-2+C|\mu_n|^{(q-2)\wedge 1}) L \right) u^{q-1} \\
			& \quad + \frac{(q-1)}{q}(2(q-2)-(3q-4)C|\mu_n|^{(q-2)\wedge 1}) K u^{2q-3}\,.
		\end{align*}
		If we choose
		$$
		K\coloneqq \frac{ M q}{(2(q-2)-(3q-4)C|\mu_n|^{(q-2)\wedge1})(q-1)}\,,
		$$
		then $\frac{(q-1)}{q}(2(q-2)-(3q-4)C|\mu_n|^{(q-2)\wedge 1}) K u^{2q-3} = f^{(0)}$. Further, if we choose
		\begin{align*}
			L\coloneqq & - \frac{2q}{3q-2} \beta^{q-2} K\,,
		\end{align*}
		then, using the second identity in \eqref{ident} and the functional equation of the gamma function,
		$$
		\langle u^{q-1}, g \rangle_{L^2(\R)} = 0
		\qquad\text{and}\qquad
		\langle u^{q-2}\partial_s u, g \rangle_{L^2(\R)} = 0 \,.
		$$
		Here the second identity follows immediately by parity. Thus, we have found a solution $g\in H^2(\R)\cap\ran\Pi$ of \eqref{eq:energy0proof} for certain explicit but irrelevant constants $c_3$ and $c_4$. (Indeed, $c_4=0$.)
		
		It follows that
		\begin{align*}
			E^{(0)} & = - \langle  \Pi f^{(0)}, (\Pi h_n^{(0)}\Pi)^{-1}  \Pi f^{(0)} \rangle_{\ran\Pi} = - \langle f^{(0)}, g \rangle_{L^2(\R)} = - M \int_\R u^{2q-3} \left( K u^{q-1}+Lu\right) \mathrm{d}s \\
			& = - M K \frac{ \beta^{3q-4}} {\alpha} \left(  \frac{\Gamma\left(\frac{3q-4}{q-2}\right)\,\sqrt\pi}{\Gamma\left(\frac{3q-4}{q-2}+\frac12\right)} - \frac{2q}{3q-2} \frac{\Gamma\left(\frac{2q-2}{q-2}\right)\,\sqrt\pi}{\Gamma\left(\frac{2q-2}{q-2}+\frac12 \right)} \right) \\
			& = -M K\frac{ \beta^{3q-4}} {\alpha} \frac{\Gamma\left(\frac{3q-4}{q-2}\right)\,\sqrt\pi}{\Gamma\left(\frac{3q-4}{q-2}+\frac12\right)} \frac{(q-2)^2}{2(q-1)(3q-2)} \,.
		\end{align*}
		Here we used the second identity from \eqref{ident} and the functional equation of the gamma function. If we insert the asymptotics $K=\frac{  M q}{2(q-2)(q-1)}+\mathcal{O}_{n\to\infty}(|\mu_n|^{(q-2)\wedge1})$ and recall the definition of $M$, we arrive at the assertion of the lemma.
	\end{proof}
	
	%%%%%%%%%%%%%%%%%%%%%%%%
	
	\subsection{Proof of Proposition \ref{prop3}}
	
	Without changing the notation, we restrict ourselves to a subsequence along which the liminf on the left side of \eqref{finalest} is realized. We first consider the simple case where (along that chosen subsequence) $\liminf_{n\to\infty} \mathcal F(u_n)/\dist(u_n,\mathcal M)^2>0$. Then, since by assumption $\mathcal F(u_n)\to 0$, we have $\dist(u_n,\mathcal M)\to 0$. Thus, $\|u_n\|^2/\dist(u_n,\mathcal M)^2\to\infty$, and therefore the left side of \eqref{finalest} is equal to $\infty$.
	
	We now consider the second case where $\liminf_{n\to\infty} \mathcal F(u_n)/\dist(u_n,\mathcal M)^2=0$. Then, after passing to a subsequence, \eqref{van} is satisfied, and we may use the expansion from Lemma~\ref{lem1}, which involves the two terms $(A)$ and $(B)$. The term $(B)$ was computed in \eqref{Bcomp}. Combining this with the bounds on $(A)$ from Lemma \ref{lem2}, \ref{lem5new}, and \ref{lem8new}, we obtain
	\begin{align*}
		\mu_n^{-4} \lambda_n^{-2} \mathcal F(u_n) \geq 
		\frac{\beta^{3q-4}}{\alpha}\frac{|\mathbb S^{d-1}|}{d^2}\frac{\Gamma \left(\frac{3q-4}{q-2}\right)\sqrt \pi}{\Gamma \left(\frac{3q-4}{q-2}+\frac{1}{2}\right)}  \frac{q(5q-6)(q-1)}{2(3q-2)} J(q,d) + \mathcal O_{n\to\infty}(|\mu_n|^{(q-2)\wedge 1}) \,,
	\end{align*}
	where, abbreviating $\xi=\mathfrak b-\mathfrak a$ as in Lemma \ref{lem8new}, $J(q,d)$ is defined by
	\begin{equation}\label{bestconst}
		J(q,d)\coloneqq \frac{(q-2)(3q-2)}{2(5q-6)} \left(\frac{3q-4}{4(q-1)}- \frac{(q-3)d}{q(d+2)}-
		\frac{d-1}{d+2}\sum_{k=0}^\infty \frac{P(k-\xi)- P(k)}{P(-1)}\right).
	\end{equation}
	Meanwhile, as a consequence of the results from Proposition~\ref{prop1} and \ref{prop2}, we find that
	\begin{align*}
		\mu_n^{-4}\lambda_n^{-2}&=\frac{(\|\omega_d u^{\sfrac{q}{2}}\|^2+\|R_n\|^2)^2}{\lambda^2_n \|r_n\|^4}=\frac{\|\omega_d u^{\sfrac{q}{2}}\|^4}{\|u\|^2_q\|u\|_q^{q-2}}\frac{C_{a,b}}{\operatorname{dist}(u_n,\mathcal M)^4}(1+o_{n\to\infty}(1))
		\\&=\frac{\beta^{3q-4}}{\alpha}\frac{|\mathbb S^{d-1}|}{d^2}\frac{\Gamma\left(\frac{2q-2}{q-2}\right)\sqrt{\pi}}{\Gamma\left(\frac{2q-2}{q-2}+\frac{1}{2}\right)}\frac{2q(q-1)^2}{(3q-2)}\frac{\|u_n\|^2}{\operatorname{dist}(u_n,\mathcal M)^4}(1+o_{n\to\infty}(1))\,.
	\end{align*}
	If we combine the previous two expansions, we arrive at the lower bound in \eqref{finalest}.
	
	Let us show that this bound is best possible in the sense that there is a sequence $(u_n)_n$ with the same properties as before, for which \eqref{finalest} is an equality. Indeed, it suffices to take
	$$
	u_n = \|u\|_q^{-1} (u + \mu_n(u^{\sfrac q2}\omega_d + \mu_n(g^{(0)} Y_{0,0} + g^{(2)} Y_{2,0})))
	$$
	for an arbitary sequence $(\mu_n)_n$ tending to $0$. Here $g^{(l)}$ are the functions in $H^1(\R)$ for which the infima in the definition of $E^{(l)}$ are attained; see the proofs of Lemma \ref{lem5new} and \ref{lem8new}. With this choice of $(u_n)_n$, the bounds in Lemma \ref{lem2}, \ref{lem5new}, and \ref{lem8new} become saturated, and therefore for this sequence we have equality in \eqref{finalest}, as claimed.
	
	To complete the proof of the proposition, it remains to prove that $J(q,d)>0$. This is equivalent to proving that $\tilde J(q,d)>0$, where
	\begin{align}
		\tilde J(q,d) \coloneqq
		\frac{(3q-4)(d+2)}{4(q-1)(d-1)}- \frac{(q-3)d}{q(d-1)}-
		\frac{1}{P(-1)}\sum_{k=0}^\infty (P(k-\xi)- P(k))\,. \label{sumest}
	\end{align}
	We distinguish two cases.
	
	\emph{Case 1.} To bound $\tilde J(q,d)$, let us first assume that $d>2, 2<q<2^*$ or $d=2$, $2.8<q <2^*$. Then Lemma \ref{lem9} below is applicable, and we infer that $P$ is strictly convex on the interval $[-1,\infty)$. We deduce that
	\begin{align}\notag
		\sum_{k=0}^\infty \frac{\left(P(k-\xi)- P(k)\right)}{-\xi}&>\sum_{k=0}^\infty \frac{\left(P(k-1)- P(k)\right)}{-1} = -P(-1)\\&>\frac{P(-1)}{-\xi} \left(\frac{(3q-4)(d+2)}{4(q-1)(d-1)}- \frac{(q-3)d}{q(d-1)}\right),\label{convexarg}
	\end{align}
	where we used
	\begin{align}
		\xi-\left(\frac{(3q-4)(d+2)}{4(q-1)(d-1)}- \frac{(q-3)d}{q(d-1)}\right)&<1-\left(\frac{(3q-4)(d+2)}{4(q-1)(d-1)}- \frac{(q-3)d}{q(d-1)}\right)\notag\\&
		=\frac{5(d-2)q^2-4(4d-3)q+12d}{4q(q-1)(d-1)}<0\,.\label{small0}
	\end{align}
	Note that
	\begin{equation}\label{xi}
		\xi \in \left(\frac{d-2}{d-1},1\right)
		,
	\end{equation}
	which is a consequence of the strict monotonicity of $\xi$ with respect to $q$ and the values of $\xi$ at the boundary points $q=2$ and $q=2^*$. The last step in \eqref{small0} follows by computing the roots of the numerator, which is quadratic in $q$ for $d>2$. The roots are $\frac{6}{5}$ and $\frac{2d}{d-2}$, and thus for all admissible $q$ the quadratic term lies below $0$. In case $d=2$, the numerator reduces to $-20q+24$, which is smaller than $0$.
	Multiplying \eqref{convexarg} by $-\xi P(-1)^{-1}<0$, we showed that $\tilde J(q,d)>0$ for the assumed parameter range of $(q,d)$. 
	
	\emph{Case 2.} We now consider the remaining case $d=2$, $2<q\leq2.8$ and give a lower bound on $E^{(2)}$ that is not optimal but more explicit than the optimal bound from Lemma \ref{lem8new}. In fact, we give this bound in general dimension $d$ and only specify to $d=2$ at the end. The argument is based on a rough version of the `completion of the square'-argument that we have used above. We recall from the proof of Lemma \ref{lem2} or Lemma \ref{lem8new} that the operator $h^{(2)}_n$ is bounded from below by $d+1+ \mathcal O(|\mu_n|^{(q-2)\wedge 1})$. Thus, for all sufficiently large $n$ the operator is invertible, and the operator-$L^2(\R)$-norm of $(h_n^{(2)})^{-1}$ is bounded by $(d+1)^{-1}+\mathcal O(|\mu_n|^{(q-2)\wedge 1})$. Moreover, we have seen that $E^{(2)}= - \langle f^{(2)}, (h_n^{(2)})^{-1} f^{(2)} \rangle_{L^2(\R)}$. Hence, we can bound, using the second identity in \eqref{ident},
	\begin{align*}
		E^{(2)} \geq&  - \|(h_n^{(2)})^{-1}\|_{op}\|f^{(2)}\|^2_{L^2(\R)}
		\\\geq &- \frac{\beta^{3q-4}}{\alpha}\frac{|\mathbb S^{d-1}|}{4d^2}\frac{\Gamma \left(\frac{3q-4}{q-2}\right)\sqrt \pi}{\Gamma \left(\frac{3q-4}{q-2}+\frac{1}{2}\right)}q(q-1)(q-2)\frac{(d-1)}{(d+2)}\left(\frac{8(q-1)(3q-4)(d-1)}{(q+2)(7q-10)(d+1)}\right)\\&+\mathcal O(|\mu_n|^{(q-2)\wedge 1})\,.
		%\\&\geq - \frac{\beta^{3q-4}}{\alpha}\frac{\Gamma \left(\frac{3q-4}{q-2}\right) \pi^{\sfrac{3}{2}}}{\Gamma \left(\frac{3q-4}{q-2}+\frac{1}{2}\right)}\left(\frac{q(q-1)^2(q-2)(3q-4)}{12(q+2)(7q-10)}\right)+\mathcal O(|\mu_n|^{(q-2)\wedge 1}).
	\end{align*}
	Now assume $d=2$, $2<q\leq 2.8$. Then this lower bound directly implies $\tilde J(q,2)>0$ in \eqref{sumest} as
	\begin{equation}\label{larger0}
		\frac{3q-4}{q-1}-2\frac{q-3}{q}-\frac{8}{3}\frac{(q-1)}{(q+2)}\frac{(3q-4)}{(7q-10)}\geq 2+\frac{1}{7}-\frac{8}{3}\cdot\frac{3}{8}\cdot\frac{1}{2}>0\,,
	\end{equation}
	where we estimated every fraction using its monotonicity in $q$ by its value at $q=2$ or $q=2.8$. (We note in passing that the above argument works as long as the left side of \eqref{larger0} is positive, that is, $q$ is smaller than approximately $57.325$.) This completes our discussion of Case 2 and therefore the proof of the proposition.
	\qed
	
	%%%%%%%%%%%%%%%%%%%%%%%%%%%
	%%%%%%%%%%%%%%%%%%%%%%%%%%%
	
	\section{Solving a certain inhomogeneous second order equation}\label{sec:solving}
	
	To complete the proof of our main theorem, we still need to prove Lemma \ref{lem8new}, as well as show the strict convexity of $P$ that we used in the proof of Proposition \ref{prop3}. This will be accomplished in the present section.
	
	\subsection{Proof of Lemma \ref{lem8new}}
	
	In the previous section, we have reduced the proof of Lemma \ref{lem8new} to solving the equation
	$$
	h_n^{(2)} g = f^{(2)}
	$$
	for $g\in H^2(\R)$ and computing
	$$
	E^{(2)} = - \langle f^{(2)}, (h_n^{(2)})^{-1} f^{(2)} \rangle_{L^2(\R)} = - \langle f^{(2)}, g \rangle_{L^2(\R)} \,.
	$$
	The operator $h_n^{(2)}$ is defined in \eqref{eq:defhnl}. The solution $g$ depends on $n$ (as well as $q$ and $d$). We stress that we already know the existence and uniqueness of this $g$ thanks to the invertibility of $h_n^{(2)}$. Since $f^{(2)}$ is even, so is $g$. This motivates us to study the equation on the positive halfline.
	
	\begin{lemma}[The degree $2$ solution]\label{lem6}
		There are $\eta\in\R$ and $(A_k)_k,(B_k)_k\subset\R$, depending on $\mu_n$, $q$, and $d$, such that the affine space of $H^2((0,\infty))$-solutions $g$ of
		\begin{equation}
			\label{ode1}
			(1-C|\mu_n|^{(q-2)\wedge1})(-\partial_s^2g +(2d+\Lambda)g)-(q-1)u^{q-2}g = f^{(2)} 
			\qquad\text{in}\ (0,\infty) 
		\end{equation}
		is parametrized by
		\begin{equation}\label{gn2}
			g =\tau u^{\sqrt{1+\frac{2d}{\Lambda}}} \sum_{k=0}^\infty A_k \cosh^{-2k}(\alpha\,\  \cdot\ )-\eta  f^{(2)}\sum_{k=0}^\infty B_k \cosh^{-2k}(\alpha\, \ \cdot\ )
		\end{equation}
		with an arbitrary parameter $\tau\in\R$. Moreover, we have
		\begin{align*}
			\lim_{s\to 0^+} \partial_s g(s) & = -2\alpha \tau \beta^{\sqrt{1+\frac{2d}{\Lambda}}} \sqrt\pi \frac{\Gamma(2\mathfrak a+1)}{\Gamma(\mathfrak a+\mathfrak b-1)} \frac{\Gamma(1)}{\Gamma\left(\frac32-\mathfrak b + \mathfrak a\right)} P_n \\
			& \quad + 2\alpha\eta M \beta^{2q-3} \sqrt{\frac{2(d-1)}{d+2}} \sqrt\pi \frac{\Gamma(\mathfrak a+\mathfrak b+1)}{\Gamma(2\mathfrak b-1)} \frac{\Gamma(1+\mathfrak b - \mathfrak a)}{\Gamma\left(\frac32\right)} Q_n \,,
		\end{align*}
		where $\mathfrak a$ and $\mathfrak b$ are defined in \eqref{xiab}, and where $P_n$ and $Q_n$ are certain numbers satisfying
		$$
		P_n = 1+ \mathcal O_{n\to \infty}(|\mu_n|^{(q-2)\wedge1}) \,,
		\qquad
		Q_n = 1+ \mathcal O_{n\to \infty}(|\mu_n|^{(q-2)\wedge1}) \,.
		$$
	\end{lemma}
	
	The fact that the solution space is one-dimensional is not surprising; since the equation is of second order, the $H^2$-requirement imposes a `boundary condition' at infinity, while there is no boundary condition at the origin. For the proof of Lemma \ref{lem8new}, we will impose the condition $\lim_{s\to 0^+} \partial_sg(s)=0$ so that $g$ extends by even reflection to an $H^2$-function on $\R$. This determines the parameter $\tau$ uniquely.
	
	\begin{solution}
		\emph{Step 1.} In this step, we explain the overall idea of the proof and defer the rigorous justification of the manipulations to the next step.
		
		The ansatz \eqref{gn2} for $g$ is a difference of two power series -- a homogeneous and an inhomogeneous formal solution of \eqref{ode1}. The coefficients $\eta$, $A_k$, and $B_k$ can be found by inserting the homogeneous and the inhomogeneous ansatz into the respective equation and comparing coefficients, as we explain now. Here (and only here) we use the symbol $h_n^{(2)}$ not in the operator-theoretic sense but rather to denote the natural differential expression associated to it. For $\mathfrak c\in\{\mathfrak a, \mathfrak b\}$ we find that
		\begin{align*}
			h_n^{(2)}( \cosh^{-2(k+\mathfrak c)}&(\alpha s ))\\=&((1-C|\mu_n|^{(q-2)\wedge 1})(-4(k+\mathfrak c)^2\alpha^2\tanh^2+2(k+\mathfrak c)\alpha^2 \cosh^{-2}+\Lambda+2d)\\&-(q-1)\beta^{q-2}\cosh^{-2})(\alpha s)\cosh^{-2(k+\mathfrak c)}(\alpha s)
			%	\\=& (1-C|\mu_n|^{(q-2)\wedge 1})(-4(k+\mathfrak c)^2\alpha^2+\Lambda+2d)\cosh^{-2(k+\mathfrak c)}(\alpha s)\\&+((1-C|\mu_n|^{(q-2)\wedge 1})(4(k+\mathfrak c)^2\alpha^2+2(k+\mathfrak c)\alpha^2) -(q-1)\beta^{q-2})\cosh^{-2(k+1+\mathfrak c)}(\alpha s)
			\\=& -(1-C|\mu_n|^{(q-2)\wedge 1})4\alpha^2 (k(k+2\mathfrak c)+\mathfrak c^2-\mathfrak a^2)\cosh^{-2(k+\mathfrak c)}(\alpha s)\\&+\alpha^2\left((1-C|\mu_n|^{(q-2)\wedge 1})(k+\mathfrak c)(4(k+\mathfrak c)+2) -\frac{2q(q-1)}{(q-2)^2}\right)\cosh^{-2(k+1+\mathfrak c)}(\alpha s)\,.
		\end{align*}
		Thus, with $C_k\in\{A_k, B_k\}$, a formal, termwise application of $h_n^{(2)}$ yields
		\begin{align}
			h_n^{(2)}&\left( \sum^\infty_{k=0} C_k \cosh^{-2(k+\mathfrak c)}(\alpha s )\right)=%M\sqrt{\frac{2(d-1)}{d+2}} u^{q-2}\\
			\sum_{k=0}^\infty C_k h_n^{(2)}( \cosh^{-2(k+\mathfrak c)}(\alpha s ))\notag\\=&
			\alpha^2(1-C|\mu_n|^{(q-2)\wedge 1})	\Bigg( \sum_{k=0}^\infty\cosh^{-2(k+\mathfrak c)}(\alpha s )\Bigg( -4C_k(k(k+2\mathfrak c)+\mathfrak c^2-\mathfrak a^2)\notag\\&+C_{k-1}\left((k-1+\mathfrak c)(4(k-1+\mathfrak c)+2) -\frac{2q(q-1)}{(1-C|\mu_n|^{(q-2)\wedge 1})(q-2)^2}\right)\Bigg)\Bigg),
			\label{inter}
		\end{align}
		where we performed an index shift in the last step. Here we use the convention $C_{-1}\coloneqq0$. If $\mathfrak c=\mathfrak a$, we set expression \eqref{inter} equal to $0$ to determine $C_k=A_k$ for the homogeneous formal solution by comparing the coefficients of $\cosh^{-2(k+\mathfrak a)}$ for each $k\in\N_0$. Similarly, for the inhomogeneous part, if $\mathfrak c=\mathfrak b$, we set \eqref{inter} equal to $-\eta^{-1}\cosh^{-2 \mathfrak b}$ and compare the coefficients of $\cosh^{-2(k+\mathfrak b)}$ for each $k\in\N_0$ to find $C_k=B_k$. In summary, this leads to the recursion relations 
		\begin{equation}\label{AkBk}
			A_k \coloneqq A_{k-1}G(k)=A_0\prod_{j=1}^k G(j) 
			\qquad\text{and}\qquad
			B_k\coloneqq B_{k-1}G(k+\mathfrak b-\mathfrak a)= B_0\prod_{j=1}^k G(j+\mathfrak b-\mathfrak a)
		\end{equation}
		with
		$$
		G(k) \coloneqq \frac{(1-C|\mu_n|^{(q-2)\wedge1})( k+\mathfrak a-1)(2(k+\mathfrak a)-1)(q-2)^2-(q-1)q}{2(1-C|\mu_n|^{(q-2)\wedge1})k(k+2\mathfrak a)(q-2)^2}\,.
		$$
		Note that the coefficient of $\cosh^{-2k}$ vanishes in case $\mathfrak c=\mathfrak a$, which is the reason for $A_0$ being freely choosable. Put differently, fixing $A_0\coloneqq1$, we obtain an additional degree of freedom $\tau$. In contrast, $B_0$ is determined by the inhomogeneity, or equivalently, if we set $B_0\coloneqq1$, $\eta$ is fixed and given by
		\begin{equation}\label{eta}
			\eta\coloneqq\frac{1}{4(1-C|\mu_n|^{(q-2)\wedge1})\alpha^2 \left( \mathfrak b^2-\mathfrak a^2\right)} \,.
		\end{equation}
		
		However, in the computation \eqref{inter} the interchange of derivative and infinite sum has to be justified. First, note that the infinite sums given in \eqref{gn2} converge absolutely, and thus we can rearrange the terms of both sums. This is a consequence of $A_k\propto k^{-\sfrac{3}{2}}$ and $B_k\propto k^{-\sfrac{3}{2}}$, which we will show in the next step. As the terms in the infinite sums still decay exponentially after differentiation if $s\in(0,\infty)$, the termwise differentiated series converges uniformly in $s$ on every open interval $I\subset (0,\infty)$ that does not contain $0$ at its boundary, and hence derivatives of the proposed $g$ are well-defined. In particular, the formal solution is a classical, smooth solution on $(0,\infty)$. The convergence $A_k,B_k\propto k^{-\sfrac 32}$ implies that $g$ extends continuously to the origin. It follows that
		$$
		(1-C|\mu_n|^{(q-2)\wedge1})(2d+\Lambda)g-(q-1)u^{q-2}g - f^{(2)} \in L^2((0,\infty)) \,,
		$$
		and therefore $g\in H^2((0,\infty))$, as claimed. (Note that $g\in H^2((0,\infty))$ implies that $g'$ extends continuously to the origin, which is not clear from the series representation.)
		
		\emph{Step 2.} We now justify the procedure outlined in Step 1, that is, we determine the asymptotics for $A_k$. The asymptotics for $B_k$ is analogous. We recall that, for $A_0\coloneqq B_0\coloneqq 1$ and $k\in\N$, the coefficients $A_k$ and $B_k$ were defined recursively in \eqref{AkBk} and $\eta$ in \eqref{eta}.
		
		Our analysis will be more precise than what is needed in Step $1$ but will be useful in the proof of the next lemma. We factor
		$$
		G(k) = G_0(k) \rho_n(k)
		\qquad\text{with}\qquad
		G_0(k) \coloneqq \frac{ \left( k+\mathfrak a+\mathfrak b-2\right) \left(k+\mathfrak a-\mathfrak b+\frac{1}{2}\right)}{k ( k+2\mathfrak a) }
		$$
		and note that, with $\mathcal O_{n\to\infty}$ describing a $k$-independent error, 
		$$
		\rho_n(k) = 1+ \frac{\mathcal O_{n\to \infty}(|\mu_n|^{(q-2)\wedge1})}{k^2} \,.
		$$
		Since there is a $C'>0$ such that \[\left|\log\left(1+\frac{\mathcal O_{n\to \infty}(|\mu_n|^{(q-2)\wedge1})}{j^2}\right)\right|\leq  \frac{C' }{j^2}|\mu_n|^{(q-2)\wedge1}\]
		for $n$ large enough and all $j\in\N$, the series $\sum_{j\in\N} \log(\rho_n(j))$ converges absolutely, and hence $\prod_{j\in\N} \rho_n(j)$ converges by continuity of the exponential function.
		It follows that for each $n$, $P_n:= \lim_{k\to\infty} \prod_{j=1}^k \rho_n(j)$ exists, is non-zero, and satisfies $P_n =1 + \mathcal O_{n\to \infty}(|\mu_n|^{(q-2)\wedge1})$. Moreover, for any $k$,
		$$
		\prod_{j=1}^k \rho_n(j) = P_n\left( 1 + \frac{\mathcal O_{n\to \infty}(|\mu_n|^{(q-2)\wedge1})}{k} \right).
		$$
		Concerning the main term, we note that
		\begin{equation}\label{Aknew}
			\prod_{j=1}^k G_0(j)=\frac{(\mathfrak a+\mathfrak b-1)_k}{(2\mathfrak a+1)_k} \frac{\left(\mathfrak a-\mathfrak b+\frac{3}{2}\right)_k}{\left(1\right)_k} \,,
		\end{equation}
		with $(\ \cdot \ )_k=\Gamma(\ \cdot+k) (\Gamma(\ \cdot\ ))^{-1}$ being the Pochhammer symbol for $k\in\N$. The crucial ingredient will be the following asymptotics for ratios of two Gamma functions
		\begin{equation}\label{Poch}
			k^{d_1-d_2}\frac{\Gamma(k+d_2)}{\Gamma(k+d_1)}=  1+\frac{(d_2-d_1)(d_1+d_2-1)}{2k}+\mathcal O_{k\to\infty}(k^{-2})\,,
			\qquad d_1,d_2>0\,,
		\end{equation}
		which can be found by Stirling's approximation. Applying \eqref{Poch} to \eqref{Aknew}, we obtain
		\begin{equation*}
			\prod_{j=1}^k G_0(j)= k^{-\sfrac{3}{2}}\frac{\Gamma(2\mathfrak a+1)}{\Gamma(\mathfrak a+\mathfrak b-1)}  \frac{\Gamma\left(1\right)}{\Gamma\left(\frac{3}{2}-\mathfrak b+ \mathfrak a \right)}+\mathcal O_{k\to\infty}(k^{-\sfrac{5}{2}}) \,. %\label{asympAnew}
		\end{equation*}
		Combining this with the behavior of $\rho_n(k)$ yields
		\begin{equation}
			A_k = k^{-\sfrac{3}{2}}\frac{\Gamma(2\mathfrak a+1)}{\Gamma(\mathfrak a+\mathfrak b-1)}  \frac{\Gamma\left(1\right)}{\Gamma\left(\frac{3}{2}-\mathfrak b+\mathfrak a \right)} P_n + \mathcal O_{k\to\infty}(k^{-\sfrac{5}{2}}) \label{asympA}
		\end{equation}
		with an error that is uniform in $n$. The asymptotics for $B_k$ is similar but with a shift by $\mathfrak b-\mathfrak a$ in every Gamma function and with a possibly different constant $Q_n$ that also satisfies $Q_n = 1 + \mathcal O_{n\to \infty}(|\mu_n|^{(q-2)\wedge1})$. Using this behavior of the coefficients $A_k$ and $B_k$, the manipulations in Step 1 can be justified, and thus we proved the assertion of the lemma concerning the solution.
		
		\emph{Step 3.} We finally compute $\lim_{s\to 0^+} \partial_s g(s)$. First, we note that for $s>0$ we can differentiate termwise the series defining $g$ and obtain
		\begin{align*}
			\partial_s g & = -2 \alpha \tau u^{\sqrt{1+\frac{2d}{\Lambda}}} \sum_{k=0}^\infty k A_k \cosh^{-2k}(\alpha\  \cdot\ ) \tanh(\alpha\ \cdot\ ) \! +\! 2\alpha \eta  f^{(2)} \sum_{k=0}^\infty k B_k \cosh^{-2k}(\alpha \ \cdot\ ) \tanh(\alpha\ \cdot\ ) \\
			& \quad + h \,,
		\end{align*}
		where $h$ involves a series that converges absolutely on all of $[0,\infty)$ and satisfies $\lim_{s\to 0^+} h(s)=0$. Since $u$ and $f^{(2)}$ are continuous at the origin, the claimed formula will follow from the fact that
		\begin{equation}
			\label{eq:abel1}
			\lim_{\sigma\to 0^+} \tanh(\sigma) \sum_{k=0}^\infty D_k \cosh^{-2k}(\sigma) = \sqrt\pi \lim_{k\to\infty} k^{\sfrac12} D_k \,,
		\end{equation}
		provided the $D_k$ are non-negative, and the limit on the right side exists. (Indeed, we apply this with $D_k\in\{kA_k,kB_k\}$, recalling the asymptotics of $A_k$ and $B_k$ from \eqref{asympA}.)
		
		To prove \eqref{eq:abel1}, we set $\cosh^{-2}(\sigma) \eqqcolon e^{-\zeta}$ and notice that the assertion is equivalent to
		$$
		\lim_{\zeta\to 0^+} \zeta^{1/2} \sum_{k=0}^\infty D_k e^{-k\zeta} = \sqrt\pi \lim_{k\to\infty} k^{\sfrac12} D_k \,.
		$$
		The latter is a consequence of a well-known Abelian theorem corresponding to the measure $\mu=\sum_{k=0}^\infty D_k \delta_k$ on $[0,\infty)$, noting that
		$$
		\mu([0,a)) = \sum_{k<a} D_k \sim 2 a^{1/2} \lim_{k\to\infty} k^{\sfrac12} D_k
		\qquad\text{as}\ a\to\infty \,.
		$$
		A simple proof of the Abelian theorem (which is an application of the dominated convergence theorem) can be found, for instance, in \cite[Theorem 10.2]{Si}.
	\end{solution}
	
	Everything is now in place to compute $E^{(2)}$ and thus to complete the proof of Lemma \ref{lem8new}.
	
	\begin{proof}[Proof of Lemma \ref{lem8new}]
		Given the notation from Lemma \ref{lem6}, we fix the value $\tau=\tau_n$ in such a way that $\lim_{s\to 0^+} \partial_s g(s)=0$, for then the even extension of $g$ is the unique $H^2(\R)$-solution of $h_n^{(2)}g=f^{(2)}$. This value is given by
		$$
		\tau_n = \sqrt{\frac{2(d-1)}{d+2}}M \eta\beta^{\xi(q-2)} \frac{2\Gamma(\mathfrak a + \mathfrak b-1)\Gamma\left(\mathfrak a-\mathfrak b+\frac{3}{2}\right)\Gamma(\mathfrak a + \mathfrak b+1)\Gamma(\mathfrak b- \mathfrak a +1)}{\sqrt{\pi}\Gamma(2\mathfrak b-1)\Gamma(2\mathfrak a+1)} \frac{Q_n}{P_n} \,.
		$$
		Since $Q_n/P_n = 1+ \mathcal O_{n\to \infty}(|\mu_n|^{(q-2)\wedge1})$ and $\eta = (4\alpha^2(\mathfrak b^2-\mathfrak a^2))^{-1}(1+ \mathcal O_{n\to \infty}(|\mu_n|^{(q-2)\wedge1}))$, we obtain the asymptotic behavior
		\begin{align*}
			\tau_n & = \sqrt{\frac{2(d-1)}{d+2}} \frac{M}{4\alpha^2(\mathfrak b^2 - \mathfrak a^2)} \beta^{\xi(q-2)} \frac{2\Gamma(\mathfrak a + \mathfrak b-1)\Gamma\left(\mathfrak a-\mathfrak b+\frac{3}{2}\right)\Gamma(\mathfrak a + \mathfrak b+1)\Gamma(\mathfrak b- \mathfrak a +1)}{\sqrt{\pi}\Gamma(2\mathfrak b-1)\Gamma(2\mathfrak a+1)} \notag \\
			& \quad \times \left( 1+ \mathcal O_{n\to \infty}(|\mu_n|^{(q-2)\wedge1}) \right).
			%\label{tau}
		\end{align*}	
		
		We write
		\begin{equation*}
			\tau_n \eqqcolon \sqrt{\frac{2(d-1)}{d+2}}M \eta\beta^{\xi(q-2)} \tilde\tau_n\,.
		\end{equation*}
		Inserting \eqref{gn2} from Lemma \ref{lem6} and $\tau_n$ into the expression for $E^{(2)}$ gives
		\begin{alignat*}{2}
			E^{(2)}&=&&-\int_{\R} g  f^{(2)}\ \mathrm{d}s
			\\&=&&-\frac{\beta^{3q-4}}{\alpha}\frac{|\mathbb S^{d-1}|}{d^2}q(q-1)^2\frac{(d-1)}{(d+2)}\alpha^2\eta
			\sum_{k=0}^\infty \left(\tilde \tau_n A_k \frac{\Gamma\left(k+\mathfrak a+\mathfrak b\right)\sqrt{\pi}}{\Gamma\left(k+\mathfrak a+\mathfrak b+\frac{1}{2}\right)}-  B_k \frac{\Gamma\left(k+2\mathfrak b\right)\sqrt{\pi}}{\Gamma\left(k+2\mathfrak b+\frac{1}{2}\right)}\right)
			\\&=&&-\frac{\beta^{3q-4}}{\alpha}\frac{|\mathbb S^{d-1}|}{4d^2}\frac{\Gamma \left(\frac{3q-4}{q-2}\right)\sqrt \pi}{\Gamma \left(\frac{3q-4}{q-2}+\frac{1}{2}\right)}q(q-1)(q-2)\frac{(d-1)}{(d+2)}
			\frac{1}{P(-1)}\sum_{k=0}^\infty \left(P(k-\xi)- P(k)\right)\\& &&+ \mathcal O_{n\to\infty}(|\mu_n|^{(q-2)\wedge1})\,.
		\end{alignat*}
		In the penultimate step, we interchanged integral and infinite sum, as the summands in the respective infinite sum have the same sign, and applied the second identity from \eqref{ident}. Due to the asymptotics \eqref{Poch} and \eqref{asympA} and their counterparts in the inhomogeneuous setting, the infinite sums are absolutely summable, and thus we are allowed to rearrange the sums.
	\end{proof}
	
	%%%%%%%%%%%%%%%%%%%
	
	\subsection{Convexity of $P$}
	
	We recall that the function $P:[-1,\infty)\to\R$ was defined in \eqref{eq:defp}. The following property is used in the proof of Proposition \ref{prop3}.
	
	\begin{lemma}[Strict convexity of $P$]\label{lem9}
		Let $d>2, 2<q<2^*$ or $d=2$, $2.8< q <2^*$. The function $P$ is strictly convex on the interval $[-1,\infty)$.
	\end{lemma}
	
	The lemma will not cover $d=2$, $2<q\leq2.8$. Indeed, numerical computations suggest that for $d=2$ and $q$ close to $2$, $P$ fails to be convex.
	
	\begin{solution} We will show strict convexity by proving that the second derivative is positive. The idea is to split $P$ into three factors, analyze their sign and that of their first and second derivatives, and conclude that the same pattern of signs can be observed for $P$. Be aware that $P$ is a well-defined, smooth function on an open, $(q,d)$-dependent set containing $[-1,\infty)$, as Gamma functions are positive, smooth functions on the positive real axis. In particular, we can differentiate at $-1$. Before investigating the single factors, let us show how $P$ inherits the property that differentiation alters the sign from its constituents. 
		
		\emph{Step 1.} 
		Let $U\subset\R$ be an open subset and $T_j:U\to\R$, $j\in\{1,\dots,J\}$, $I$-times differentiable functions for $I\in\N_0$ and $J\in\N$ with $(-1)^iT_j^{(i)}>0,$ 
		for all $i\in\{1,\dots,I\}$ and $j\in \{0,\dots, J\}$. Here $(\ \cdot\ )^{(i)}$ denotes the $i$-th derivative. Applying the general Leibniz rule, we observe that
		\[(-1)^I\left(\prod_{j=1}^J T_j\right)^{(I)}=\sum_{m_1+m_2+\dots+m_J=I}\binom{I}{m_1,m_2,\dots,m_J} \prod_{j=1}^J (-1)^{m_j}T_j^{(m_j)}>0\] 
		since every factor $(-1)^{m_j}T_j^{(m_j)}>0$.
		
		\emph{Step 2.}
		We are now going to apply Step 1 with $I=2$ and $J=3$ to the smooth function $P$. To this end, we write 
		\[P(x)=\prod_{j=1}^{3}T_j(x)\qquad \text{ with }\qquad T_j:\left(-\frac{3}{2},\infty\right)\to \R\,,\, T_j(x)\coloneqq \frac{\Gamma(x+\epsilon_1^j)}{\Gamma(x+\epsilon_2^j)}\,,\] $j\in\{1,2,3\}$, where we denote 
		\[\begin{matrix*}[l]
			\epsilon_1^1\coloneqq\frac{3}{2}\,, & \epsilon_1^2\coloneqq2\mathfrak b-1\,, & \epsilon_1^3\coloneqq2\mathfrak b\,,\\
			\epsilon_2^1\coloneqq\mathfrak b-\mathfrak a +1\,, & 
			\epsilon_2^2\coloneqq\mathfrak b+\mathfrak a+1\,, & 
			\epsilon_2^3\coloneqq 2\mathfrak b +\frac{1}{2}\,.
		\end{matrix*}\]
		Note that $\frac{3}{2}\leq \epsilon_1^j< \epsilon_2^j$ for all $j\in\{1,2,3\}$ as $2^{-1}<\xi<1$ for the given parameter range of $q$ and $d$; see \eqref{xi}. In particular, $T_j$ is a well-defined, positive function on $\left(-\frac{3}{2},\infty\right)$ since the Gamma functions are only evaluated at positive entries.
		
		Next, we will use monotonicity of the Digamma function $\Psi(y)\coloneqq \partial_y(\log (\Gamma(y)))$, $y>0$,  and its derivative, the Trigamma function $\Psi_1$. The monotonicity behavior can be deduced in an elementary way from an integral formula for the Digamma function \cite[12.3]{WhWa}, for instance.
		As $\Psi$ is strictly monotonically  increasing on the positive real axis, the second inequality follows through
		\[\partial_xT_j(x)=T_j(x)\left(\Psi(x+\epsilon^j_1)-\Psi(x+\epsilon^j_2)\right)<0\]
		for $x>-\frac{3}{2}$ and $j\in\{1,2,3\}$. 
		Similarly, we compute 	\[\partial^2_x T_j(x)=T_j(x)\left(\left(\Psi(x+\epsilon^j_1)-\Psi(x+\epsilon^j_2)\right)^2+\left(\Psi_1(x+\epsilon^j_1)-\Psi_1(x+\epsilon^j_2)\right)\right)>0\]
		for $x>-\frac{3}{2}$ and $j\in\{1,2,3\}$ using that $\Psi_1$ is strictly monotonically decreasing on the positive real axis. Therefore, we can apply Step 1 to obtain $\partial_x^2 P>0$.
	\end{solution}
	
	%%%%%%%%%%%%%%%%%%%%%%%%%%%%%%

	\bibliographystyle{amsalpha}

\end{document}